\documentclass[reqno]{amsart}
\usepackage[english]{babel}  
\usepackage{url}
\usepackage{graphicx}
\usepackage{amsopn}
\usepackage{amsmath, amsfonts, amssymb, tikz-cd, amsthm,xfrac}
\usepackage{mathdots}
\usepackage{setspace}
\usepackage{mathtools}
\usepackage{multirow}
\usepackage{caption}
\usepackage[all]{xy}
\usepackage{lscape}
\usepackage{comment}
\usepackage{float}
\usepackage{enumitem}
\usepackage{fancyhdr}
\usepackage{xcolor}
\usepackage{hyperref}

\newtheorem{theorem}{Theorem}[section]
\newtheorem*{theorem*}{Theorem}
\newtheorem{thm}{Theorem}[section]
\newtheorem*{problem}{Problem}

\newtheorem{lemma}[theorem]{Lemma}
\newtheorem{proposition}[theorem]{Proposition}
\newtheorem{corollary}[theorem]{Corollary}
\theoremstyle{definition}
\newtheorem{definition}[theorem]{Definition}
\theoremstyle{remark}

\numberwithin{equation}{section}
\newcommand{\stirling}[2]{\genfrac{[}{]}{0pt}{1}{#1}{#2}}

\newcommand{\bZ}{\mathbb Z} \newcommand{\bQ}{\mathbb Q}
\newcommand{\bC}{\mathbb C} 
 \newcommand{\mF}{\mathcal F}
\newcommand{\mG}{\mathcal G}   
\newcommand{\fF}{\mathfrak F}

\DeclareMathOperator{\Hom}{Hom}

\DeclareMathOperator{\Ind}{Ind}

\DeclareMathOperator{\GL}{GL}

\DeclareMathOperator{\Gal}{Gal}

\DeclareMathOperator{\M}{M}

\DeclareMathOperator{\GSp}{GSp}

\DeclareMathOperator{\tr}{Tr}

\DeclareMathOperator{\End}{End}
\DeclareMathOperator{\Rank}{Rank}

\DeclareMathOperator{\Res}{Res}
\DeclareMathOperator{\St}{St}
\DeclareMathOperator{\Mir}{Mir}

\newcommand{\bbsm}{\left (\begin{smallmatrix}}      \newcommand{\besm}{\end{smallmatrix}\right )}
\newcommand{\bsm}{\left[ \begin{smallmatrix}}      \newcommand{\esm}{\end{smallmatrix}\right]}
\newcommand{\bbm}{\begin{pmatrix}}      \newcommand{\bem}{\end{pmatrix}}

\begin{document}
\title[Twisted Jacquet modules of cuspidal representations of $\GL(n, q)$]{Twisted Jacquet modules associated to maximal parabolic subgroups and  cuspidal representations of $\GL(n, q)$}
\author{Kumar Balasubramanian} 
\email{bkumar@iiserb.ac.in}
\address{Department of Mathematics\\
IISER Bhopal\\
Bhopal, Madhya Pradesh 462066, India}
\author{Krishna Kaipa} 
\email{kaipa@iiserpune.ac.in}
\address{Department of Mathematics\\
IISER Pune\\
Pune, Maharashtra 411008, India}
\author{Himanshi Khurana}
\email{himanshikhurana@hri.res.in}
\address{ ${}^1$ Department of Mathematics\\
Harish-Chandra Research Institute\\
Prayagraj, Uttar Pradesh 211019, India
\newline ${}^2$ Homi Bhabha National Institute\\
Training School Complex, Anushakti Nagar\\
Mumbai 400094, India}


\begin{abstract}
Let $\pi$ be a cuspidal representation of $\GL(n,F)$ over a finite field $F$. Let $P=MN$ be the Levi decomposition of a  maximal parabolic subgroup corresponding to the partition $(k,n-k)$ of $n$.  Given a  rank $r$ character $\psi_r$ of the unipotent radical $N$,  the 
twisted Jacquet module $\pi_{N, \psi_r}$ is a  representation of the subgroup $M_r$ of $M$ which stabilizes $\psi_r$. The problem we solve in this work is to determine the structure of 
$\pi_{N, \psi_r}$ as a  $M_r$-module.  This problem was first studied by D. Prasad, who solved it  for the case $r=k=n/2$. This and subsequent works on the problem for special cases of $(r,k,n)$, identify the structure of $\pi_{N, \psi_r}$ by calculating its character and matching it to a known representation of $M_r$.  In this work, we solve the problem for all values of  $(r,k,n)$ directly  without  calculating  the character of $\pi_{N, \psi_r}$. Our solution depends on two other key conceptual advances: (i) We generalize the Bernstein-Zelevinsky framework for studying  representations of the Mirabolic subgroup of $\GL(n,F)$, to maximal parabolic subgroups $P$. In particular, we show that the twisted Jacquet functor which takes a  representation of $P$ to its twisted Jacquet modules (one for each rank), gives an equivalence of categories between Rep$(P)$ and the direct sum $\oplus_r \text{Rep}(M_r)$.
(ii) Using this, we construct a pair of recursively defined representations $\Pi_{k,n}, \Pi_{n-k,n}^\dagger$ of $P$, which generalizes to $P$,  the representation of the Mirabolic subgroup obtained from the trivial representation by recursively applying  the Bernstein-Zelevinsky $\Phi^+$ functor. Like the representation $(\Phi^+)^{n-1}(1)$ of the Mirabolic subgroup, the representation $\Pi_{n-k,n}^\dagger$ satisfies a universal property with respect to restrictions to $P$ of cuspidal representations of $\GL(n,F)$. Our solution of the main problem is a simple consequence of this universal property.
\end{abstract}

\maketitle
\section{Introduction}  \label{introduction}
The study of Whittaker models for representations, in particular its multiplicity one property, has always played an important role in the representation theory of reductive groups over both finite and $p$-adic fields. However, many important representations do not admit a Whittaker model. In such cases, one considers degenerate Whittaker models (also called as the twisted Jacquet module). These models are obtained by replacing the generic character of the maximal unipotent subgroup with a more general character, often arising from the unipotent radical of a non-minimal parabolic subgroup.\\

According to the `philosophy of cusp forms', cuspidal representations serve as the building blocks for understanding the representation theory of $\GL(n,F)$. Any representation of $\GL(n,F)$ can be realized in terms of cuspidal representations of $\GL(n,F)$ and of Levi subgroups of proper parabolic subgroups. Therefore, it is of fundamental importance to understand any structural results about cuspidal representations. In the case of finite $\GL(n,F)$, the explicit construction of cuspidal representations using cohomological methods were described in the foundational work of Deligne and Lusztig. In this paper, we focus on understanding the twisted Jacquet module for cuspidal representations of $\GL(n,F)$ defined over the finite field $F$, with respect to all maximal parabolic subgroups $P$ of $\GL(n,F)$ and all characters $\psi$ of the unipotent radical of $P$. \\


\subsection{Main Problem} In this section, we recall the definition of the twisted Jacquet module, set up some notation and state the main problem we study in this paper. \\

Let $F$ be the finite field with $q$ elements and let $\GL(n,F)$ be the  general linear group. Let $P$ be a parabolic subgroup of $\GL(n,F)$ with Levi decomposition $P=MN$, where $N$ is the unipotent radical of $P$. Let $(\pi,V)$ be an irreducible complex  representation of $\GL(n,F)$ and $\psi$ be a character of $N$. For $m \in M$ and a character $\psi$ of $N$, let $\psi^{m}$ be the character of $N$ defined by $\psi^{m}(n)=\psi(m^{-1}nm)$. Let 
\[V(N,\psi)= \langle \pi(n)v-\psi(n)v \colon n \in N, v \in V \rangle_{\mathbb{C}}.\]  
It is easy to see that $V(N, \psi)$ is invariant under the subgroup $M_\psi$ of $M$ given by 
\[M_{\psi} = \{m \in M \colon {\psi}^{m}(n)=\psi(n) , \forall n \in N \}.\] It follows that we get a representation of $M_{\psi}$ in the space of $\psi$-twisted coinvariants $V_{N,\psi}=V/V(N,\psi)$. 
The representation $(\pi_{N,\psi}, V_{N,\psi})$ of $M_\psi$ is known as the \emph{twisted Jacquet module} of $\pi$ with respect to $(N,\psi)$.  It is an interesting question to determine for which irreducible representations $\pi$ of $\GL(n,F)$ and characters $\psi$ of $N$, the twisted Jacquet module $\pi_{N,\psi}$ is non-zero and to understand its structure as a representation of $M_{\psi}$.\\

We denote  $\GL(n,F)$ as $G_n$. Let $\pi$ be a cuspidal representation of $G_n$, and let $P=P_{k,n}$ be a maximal parabolic subgroup of $G_n$ corresponding to the partition $(k,n-k)$ of $n$. We have $M \cong G_k \times G_{n-k}$ and $N \cong \M_{k \times n-k}(F)$. It is easy to see that the characters of $N$ are of the form $\psi_A(X)=\psi_0(\tr(A^{\top}X))$, where $\psi_0:F \to \bC^\times$ is a non-trivial additive character of $F$ and $A \in \M_{k \times n-k}(F)$. The character $\psi_A$ is said to be of rank $r$ if the matrix $A$ has rank $r$. It is said to be non-degenerate if $A$ has full rank $\min\{k,n-k\}$, and degenerate otherwise. \\

\begin{problem}[$\text{Prob}_{k,n,r,\pi}$] Let $\pi$ be an irreducible cuspidal representation of $G_n$, and let $P=P_{k,n}$ be a maximal parabolic subgroup of $G_n$ associated with the partition $(k,n-k)$ of $n$. Let $P=MN$ be the Levi decomposition of $P$, and let $\psi$ be a rank $r$ character of $N$. Determine the structure of the  twisted Jacquet module $\pi_{N,\psi}$ as a $M_{\psi}$-module.
\end{problem}
\noindent The main result of this paper (see Theorem \ref{main_theorem}) is the solution of the above problem.

\subsection{Some Known Results} The study of $\text{Prob}_{k,n,r,\pi}$ was initiated by D. Prasad in \cite{Dip[1]}, and it is the main inspiration for this work. In  \cite{Dip[1]},  $\text{Prob}_{k,2k,k,\pi}$ was solved for the case $n=2k$ and a non-degenerate character $\psi$ of rank $k$, by first computing the character of $\pi_{N,\psi}$, and showing that it equals the character of a known representation  of $M_\psi \simeq G_k$. The result is that for a 
 cuspidal representation $\pi_\theta$ of $G_{2k}$ associated to a regular character $\theta$ of $F_{2k}^\times$:
\[ \pi_{N, \psi} \simeq \Ind_{F_k^\times}^{G_k} \theta|_{F_k^\times}. \] 
The subsequent works \cite{KumHim[4]}, \cite{KumHim[2]}, \cite{KumHimgl6rk2},\cite{KumHim[3]} solve the problem $\text{Prob}_{k,n,r,\pi}$ for $r \in \{1,2\}$ and some small values of $k,n$. The above problem has a natural generalization $\text{Prob}_{\lambda \vdash n,n,r,\pi}$ to the case when we replace the  maximal parabolic subgroup $P$ of $G_n$ with a standard parabolic subgroup for an arbitrary partition $\lambda \vdash n$ of $n$, and for an appropriate notion of rank of a character of $N$. The  special case when  $\lambda$ has $k$ parts of  equal size $n$, and the character of $N$ is  non-degenerate was solved by O. Gorodetsky and  Z. Hazan  in  \cite{HazGor}, again by first carrying out the  calculation of the character of $\pi_{N, \psi}$, and then matching it to the character of a known representation of $M_\psi \simeq G_n$.
The result is that for a 
 cuspidal representation $\pi_\theta$ of $G_{kn}$ associated to a regular character $\theta$ of $F_{kn}^\times$:
\[ \pi_{N, \psi} \simeq \St_{G_n}^{\otimes\ (k-2)} \otimes \Ind_{F_n^\times}^{G_n} \theta|_{F_n^\times}, \] 
which specializes to D. Prasad's result above for $k=2$. Here $\St_{G_{n}}$ is the Steinberg representation of $G_n$. In another direction, an analogue of the problem $\text{Prob}_{2,4,2, \pi}$ with the finite field $F$ replaced by a finite principal ideal local ring of length $2$ was recently studied in  \cite{parashar2024degeneratewhittakerspacegl4mathfrako2}. The structure of the twisted Jacquet module for representations of many finite classical groups is still unknown. In a recent work \cite{Jon}, the dimension of the module for any irreducible representation of $\GSp(4,q)$ with respect to a non-degenerate character of the unipotent radical of the Siegel parabolic subgroup was determined.\\


\subsection{Strategy of Proof} 
In our solution of $\text{Prob}_{k,n,r,\pi}$ (see Theorem \ref{main_theorem}), we obtain the structure of $\pi_{N, \psi}$ directly, bypassing the need to  carry out the elaborate calculation of  the character of $\pi_{N, \psi}$. In order to explain the idea of proof, let us assume here the condition that $k < n/2$ (which can be relaxed). Assuming this, we construct a fundamental representation $\Pi_{k,n}$ of $P_{k,n}$ which has the property that for any cuspidal representation $\pi=\pi_\theta$ of $G_n$, there is a cuspidal representation $\pi_{\tilde\theta}$ of $G_k$  such that $\pi_\theta|_{P_{k,n}} = \pi_{\tilde\theta} \otimes \Pi_{k,n}$, where $P_{k,n} = \bbsm G_k & \ast\\& G_{n-k} \besm$, and  by $\pi_{\tilde\theta}$, we mean its inflation from $G_k$ to 
$P_{k,n}$. We show that the  problem  reduces to determining the twisted Jacquet module of $\Pi_{k,n}$ with respect to the same character $\psi$ of $N$. We are immediately able to determine $\pi_{N, \psi}$ because,  the very  construction of  $\Pi_{k,n}$ (see Definition \ref{def_Pi}) is by prescribing its twisted Jacquet modules in a way that ensures the desired property $\pi_\theta|_{P_{k,n}} = \pi_{\tilde\theta} \otimes \Pi_{k,n}$ (see Theorem \ref{Pi_k_n_char} and part (1) of Theorem \ref{Pi_universal_property}). Our recursive construction of the representation $\Pi_{k,n}$ of $P_{k,n}$ depends on some new (to the best of our knowledge) ideas that we develop in this work: we extend the Bernstein-Zelevinsky framework for studying representations of the Mirabolic subgroup $\Mir_n$ of $G_n$ to the maximal parabolic subgroups. For each rank $0 \leq r \leq k$,  we define functors $\Phi_r^- \colon {\rm Rep}(P_{k,n}) \to  {\rm Rep}(M_r)$,  where $M_r:=M_{\psi_r}$ for a rank $r$ character $\psi_r$ of $N$, that takes a representation $(\tau,V)$ of $P_{k,n}$ to its twisted Jacquet module $(\tau_{N,\psi_r}, V_{N,\psi_r})$. We show that the composite 
twisted Jacquet functor $\fF^-: \text{Rep}(P_{k,n}) \to \oplus_{r=0}^k \text{Rep}(M_r)$ given by $(\Phi_0^-, \dots, \Phi_k^-)$ is an equivalence of categories.  The inverse functor $\fF^+ \colon \oplus_{r=0}^k \text{Rep}(M_r) \to \text{Rep}(P_{k,n})$ allows us to 
synthesize a representation $\tau$ of $P_{k,n}$ from its twisted Jacquet modules (see part (d) of Theorem \ref{Phi+-_properties}). We show that 
the recursive definition of the representation $\Pi_{k,n}$ (actually a closely related representation $\Pi_{n-k,n}^\dagger$) is a natural generalization from $\Mir_n$ to $P_{k,n}$ of the recursively defined representation $(\Phi^+)^{(n-1)}(1)$ of $\Mir_n$. The property of $(\Phi^+)^{(n-1)}(1)$, that the restriction to $\Mir_n$ of any cuspidal representation of $G_n$ is isomorphic to
$(\Phi^+)^{(n-1)}(1)$, generalizes to a universal property of $\Pi^\dagger_{n-k, n}$ with respect to restrictions to $P_{k,n}$ of cuspidal representations of $G_n$. It may be useful to state here our main result (Theorem \ref{main_theorem}). In order to do that, we need to introduce some more notation. Let $P^*_{r,k}$ denote the image of the parabolic $P_{r,k}$ under the inverse transpose automorphism $g \mapsto g^{-\top}$ of $G_k$. The representation $\Pi^*_{r,k}$ of $P^*_{r,k}$ is defined in terms of $\Pi_{r,k}$ by $\Pi^*_{r,k}(g)=\Pi_{r,k}(g^{-\top})$. There are epimorphisms $p_1:M_r \to P^*_{r,k}$, $p_2:M_r \to P_{r,n-k}$ and $p_3:M_r \to G_r$ which allow us to inflate the representations $\Pi^*_{r,k}$, $\Pi_{r,n-k}$ and $\Ind_{F_r^\times}^{G_r} \theta_r$ (occurring in the theorem) of  $P^*_{r,k}, P_{r,n-k}$ and $G_r$ respectively to $M_r$. 
\begin{theorem*}[Main Theorem] Let $\pi=\pi_\theta$ be a cuspidal representation of $G_n$, and let $P=MN$ be a maximal parabolic subgroup of $G_{n}$ associated with the partition $(k,n-k)$ of $n$. Let $\psi$ be a rank $r$ character of $N$. Let $d=\text{gcd}(r,k,n)$ and let $\theta_r$ be a character of $F_r^\times$ with the property that $\theta_r|_{F_d^\times} = \theta|_{F_d^\times}$.  The twisted Jacquet module $\pi_{N, \psi}$ as a representation of $M_r$ is isomorphic to 
\[ \pi_{N, \psi} \simeq \Pi^*_{r,k} \otimes \Pi_{r,n-k} \otimes \Ind_{F_r^\times}^{G_r} \theta_r. \] 
In particular,
\[ \dim  \pi_{N,\psi} =
q^{\binom{r}{2}} \prod_{i=1}^{r-1}(q^i-1) \cdot \prod_{i=r}^{k-1}(q^i-1) \cdot \prod_{i=r}^{n-k-1}(q^i-1).\]
\end{theorem*}
In the case $r=k=n/2$, our solution 
recovers D. Prasad's solution $\pi_{N, \psi} \simeq \Ind_{F_k^\times}^{G_k} \theta|_{F_k^\times}$, because 
the representations $\Pi^*_{k,k}$ and $\Pi_{k,k}$ are trivial.



\subsection{Some $p$-adic Questions} In the case when $F$ is a $p$-adic field, our work raises the following natural questions:\begin{enumerate}
    \item  Develop the appropriate version of the properties of the functors $\Phi_r^{\pm}$ given in  Theorem \ref{Phi+-_properties}. 
    \item Construct if possible,  the representation $\Pi_{k,n}$ given in Definition \ref{def_Pi}.
    \item Is there a generalization of  the universal property in part (1) of Theorem \ref{Pi_universal_property}, for the distribution character of a (suitably nice) supercuspidal representation of $\GL(n,F)$?
    \item Generalize D. Prasad's result (Theorem \ref{n=2k_nondeg}) to the $p$-adic case. A sub-question here is: for a (suitably nice) supercuspidal representation $\pi$ of $\GL(k,F)$ and the Steinberg representation $\St$ of $\GL(k,F)$, is there a concise description of the representation $\pi \otimes \St$ (analogous to part (2) of   Lemma \ref{ind=st_times_cuspidal})?
    \end{enumerate} 
We hope that understanding these questions can shed light on the  structure of the module $\pi_{N, \psi}$ for $p$-adic groups. 

\subsection{Organization} This paper is organized as follows. In \S \ref{prelim}, we set up the notation and record some preliminary representation theoretic results needed for the sequel. In \S \ref{combin}, we record some combinatorial results that we will need. In \S \ref{sec_equivalence}, we generalize the Bernstein-Zelevinsky framework for Mirabolic subgroups to the maximal parabolic subgroups $P_{k,n}$, and establish the equivalence between the categories ${\rm Rep}(P)$ and $\oplus_r {\rm Rep}(M_r)$. We then construct in \S \ref{sec_Pi_construction}, the   representation $\Pi_{k,n}$ of $P_{k,n}$ (see Definition~\ref{def_Pi}) that satisfies a universal property with respect to restriction of cuspidal representations of $G_n$ to $P_{k,n}$ (see Theorem  \ref{Pi_universal_property}). In  \S \ref{structure}, we use the representation $\Pi_{k,n}$ to obtain  our solution  of $\text{Prob}_{k,n,r,\pi}$ (see   Theorem \ref{main_theorem}).

\section{Preliminaries} \label{prelim}
We begin with a list of some symbols and notations that we frequently use.
\begin{description}
\item [$F$] A finite field of order $q$. The symbol $F_m$ denotes the finite field of order $q^m$. The symbol $F_m^\times$ denotes the multiplicative group of $F_m$.
\item [$\phi(x)=x^q$] is the Frobenius automorphism of $F_m$. The Galois group Gal$(F_m/F)$ is the cyclic group of order $m$ generated by $\phi$.
\item [$N_{F_m/F}(x)=x \phi(x) \cdots \phi^{m-1}(x)$] is the norm of $x \in F_m$ over $F$.
\item [$\tr_{F_m/F}(x)=x+\phi(x) +\dots +\phi^{m-1}(x)$] is the trace of $x \in F_m$ over $F$.
\item[$\widehat{\mathcal G}$] refers to the set of isomorphism classes of irreducible representations of a finite group $\mathcal G$. By a representation, we always mean a finite dimensional complex representation.
\item [${\rm Rep}(\mathcal G)$]  denotes the category whose objects are finite dimensional $\bC$-representations of the finite group $\mathcal G$, and whose  morphisms are the intertwining operators between the representations. 
\item [$\chi_{\rho}$] denotes the character of  a representation $\rho$ of a finite group $\mathcal G$. We implicitly use the fact that two representations of $\mathcal G$ are isomorphic if and only if their characters are equal.
\item [$\langle \chi, \chi'\rangle_\mathcal{G}$] denotes the inner product $(\sum_{g \in \mathcal G} \chi(g) \overline{\chi'(g)} )/|\mathcal G|$ of characters $\chi, \chi'$ of $\mathcal G$.
\item [$\rho|_{H}$] is the restriction of a representation $\rho$ of a group $\mathcal G$ to its subgroup $H$. More generally, given a subset $Y \subset X$ and a function $f$ on $X$, we denote the restriction of $f$ to $Y$ by   $f|_{Y}$.
\item [$\Ind_H^\mathcal{G}$] denotes the induced representation/character. We implicitly use Frobenius reciprocity: $\langle \chi', \Ind_H^\mathcal{G} \chi \rangle_\mathcal{G}=\langle \chi'|_{H}, \chi \rangle_H$.
\item[$\psi_0$]  a fixed non-trivial additive character of $F$.
\item[$\M_{k \times \ell}(F)$] the additive group (or $F$-vector space)  of $k \times \ell$ matrices over $F$.
\item[$\tr(X)$] the trace of a square matrix $X$.
\item[$X^\top$] the transpose  of a  matrix $X$.
\item [$G_n$] is the general linear group $\GL(n,F)$.
\item [$I_n$] is the identity matrix in $G_n$.
\item [$P_{r,n}=$] $\{ \bbsm g_1 & g_3\\& g_2 \besm \colon g_1 \in G_r, g_2 \in G_{n-r}, g_3 \in \M_{r \times n-r}(F)\}$ is the parabolic subgroup of $G_n$ for the partition $(r,n-r)$ of $n$.
\item [$P_{r,n}^*=$] $\{ \bbsm g_1 & \\g_3 & g_2 \besm \colon g_1 \in G_r, g_2 \in G_{n-r}, g_3 \in \M_{n-r \times r}(F)\}$ is the image of $P_{r,n}$ under the inverse transpose isomorphism of $G_n$.
\end{description}
\subsection{Regular characters of $F_n^\times$ and cuspidal representations of $G_n$} \hfill \\
Let $\gamma$ be  a generator of the cyclic group $F_n^\times$, and let $\widehat{F_n^\times}$ be the group of characters of $F_n^\times$. Let  $\imath_n \colon F_n^\times \to \widehat{F_n^\times}$ be the group isomorphism 
\[ \imath_n(\gamma^j) = \theta_j, \qquad \theta_j(\gamma^s)=\exp(\tfrac{2 \pi \sqrt{-1}  j s}{q^n-1}).\]
Let $\mF_n \subset F_n$ be the elements of $F_n^\times$ which are not contained in any proper subfield of $F_n$. Alternatively, $\mF_n\subset F_n^\times$ consists of elements whose Gal$(F_n/F)$-orbit has size $n$, where 
the Galois group $\Gal(F_n/F)$ is the cyclic group of order $n$ generated by the automorphism $\phi(x)=x^q$.  The \emph{regular characters} of $F_n^\times$,  denoted $(\widehat{F_n^\times})_{\text{reg}}$,  consist of $\imath_n(\mF_n)$. If $F_d$ is a subfield of $F_n$, then $\beta=N_{F_n/F_d}(\gamma) = \gamma^{(q^n-1)/(q^d-1)}$ is a generator of the cyclic group $F_d^\times$. Again, we have an isomorphism
\[ F_d^\times \xrightarrow{\imath_d}  \widehat{F_d^\times}, \quad \imath_d(\beta^j)=\varphi_j, \quad \varphi_j(\beta^s)=\exp(\tfrac{2 \pi \sqrt{-1}  j s}{q^d-1}). \]

\begin{lemma} \label{imath_lemma} \hfill \\
\begin{enumerate}
\item[\emph{(1)}] Any  character of $F_d^\times$ can be obtained by restricting a  character of $F_n^\times$ to $F_d^\times$. Moreover, we  have a commutative square:
 \[
\begin{tikzcd}[column sep=6em, row sep=5em]
F_n^\times 
  \arrow[r, equal, "\imath_n"] 
  \arrow[d, twoheadrightarrow, "N_{F_n/F_d}"']
& \widehat{F_n^\times}
  \arrow[d, twoheadrightarrow, "\operatorname{Res}^{F_n^\times}_{F_d^\times}"]
\\
F_d^\times \arrow[r, equal, "\imath_d"']
&
\widehat{F_d^\times}
\end{tikzcd}
\]
\item[\emph{(2)}] If $d$ is a proper divisor of $n$, then any 
character of $F_d^\times$ can be obtained by restricting a  \underline{regular} character of $F_n^\times$ to $F_d^\times$:
\[
\begin{tikzcd}[column sep=6em, row sep=5em]
\mF_n  
  \arrow[r, equal, "\imath_n"] 
  \arrow[d, twoheadrightarrow, "N_{F_n/F_d}"']
&  (\widehat{F_n^\times})_{{\rm reg}}
  \arrow[d, twoheadrightarrow, "\operatorname{Res}^{F_n^\times}_{F_d^\times}"]
\\
F_d^\times \arrow[r, equal, "\imath_d"']
&
\widehat{F_d^\times}
\end{tikzcd}
\]
    
\end{enumerate}

\end{lemma}
\begin{proof}
(1) The element $\varphi_j=(\imath_d \circ N_{F_n/F_d})(\gamma^j)$ of $\widehat{F_d^\times}$as defined above, is  clearly the restriction of  $\theta_j$ to  $F_d^\times$, i.e., $(\Res_{F_d^\times}^{F_n^\times} \circ \imath_n)(\gamma^j)$.\\

(2) Here $d$ is a proper divisor of $n$. In order to show that  any character of $F_d^\times$ can be obtained by restricting a \emph{regular} character of $F_n^\times$ to $F_d^\times$, it suffices to show the map $N_{F_n/F_d}:\mF_n \to F_d^\times$ is surjective.  Let $e$ be a proper divisor of $n$ such that $d|e$. It suffices to show $N_{F_n/F_e}(\mF_n)=F_e^\times$, because $N_{F_e/F_d}(F_e^\times)=F_d^\times$. Choosing a maximal such $e$, and replacing $d$ with $e$, we may now assume that $\dim_{F_d}(F_n)$ is a prime number $\ell$. In this case $\mF_n = F_n \setminus F_d$. For each $c \in F_d^\times$ there are $1+q+\dots +q^{\ell-1}$ solutions in $F_n$ to the equation $N_{F_n/F_d}(x)=c$ of which there are at most $\ell$ in $F_d$, therefore there are $(q-1)+ (q^2-1)+ \dots +(q^{\ell-1} -1) \geq  q-1>0$ solutions in $\mF_n$.
\end{proof}

We recall that the cuspidal representations of $G_n$ arising from regular characters $\theta=\imath_n(x), \theta'=\imath_n(x')$ of $F_n^\times$  are isomorphic if and only if $x, x' \in \mF_n$ are in the same Gal$(F_n/F)$ orbit. The Gal$(F_n/F)$-orbits on $\mF_n$ correspond (via the minimal polynomial) to the set of roots of monic irreducible polynomials of degree $n$ in $F[X]$. It is well known that the number of such polynomials (and hence, also the number of isomorphism classes of cuspidal representations of $G_n$) is $( -\delta_{1,n} + \sum_{d|n} \mu(d) q^{n/d} )/n$,
where  $\mu(n)$  is the M\"{o}bius  function, and $\delta_{ij}$ is the Kronecker delta. \\

For $a \in F_n$, let $m_a$ be the $F$-linear endomorphism of $F_n$ given by $x \mapsto a x$. Clearly $a \mapsto m_a$ is a monomorphism from $F_n^\times \to G_n$ (once a basis for $F_n/F$ is chosen). In this way, we view $F_n^\times$ as a subgroup of $G_n$. An element $g \in G_n$ is said to \emph{come from $F_n^\times$} if there exists $a\in F_n$ and a basis of $F_n/F$ such that $g$ is the  matrix of $m_a$ with respect to this basis. Equivalently, $g \in G_n$ comes from $F_n^\times$ if and only if its minimal polynomial $p(X) \in F[X]$ is an irreducible polynomial of degree $m|n$.
\begin{definition} \label{Theta_definition}
    For a character $\theta$ of $F_n^\times$, let $\Theta_\theta$ be the class function  on $G_n$ defined as follows. Let $g \in G_n$ with  $g=s \cdot u$ being the Jordan decomposition  of $g$ into its semisimple part  $s$ and unipotent part $u$. In case  $s$ comes from $F_n^\times$, let  $\lambda \in F_m$ be a root of  the irreducible minimal polynomial $p(X) \in F[X]$ of $s$, and let $t(g)=\dim_{F_n} \ker(g - \lambda I_n)$. We define:
\begin{equation} \label{eq:Theta_definition}
   \Theta_{\theta}(g)= \begin{cases} 
   0 &\text{ if $s$ does not come from $F_n^\times$} \\
   (-1)^{n-1}\left( \sum_{i=0}^{m-1} \theta(\phi^i(\lambda))\right) \,  \prod_{j=1}^{t(g)-1} (1-q^{jm}) &\text{ if $s$ comes from $F_n^\times$.} \\\end{cases} 
\end{equation}
\end{definition}
We recall the character of the cuspidal representation $\pi_\theta$ of $G_n$ associated with the regular character $\theta$ of $F_n^\times$ (see \cite{Gel[1], Dip[1]}):
\begin{thm} [J.A. Green]  
\label{character value of cuspidal representation (Dipendra)}
Let $\theta$ be a regular character of $F^{\times}_{n}$. Let $\pi=\pi_{\theta}$ be the cuspidal representation of $G_n$ associated to $\theta$. The character of $\pi_\theta$ equals the class function $\Theta_{\theta}$ given in Definition \ref{Theta_definition} above.

\end{thm}

\subsection{Characters of $N$ and the associated twisted Jacquet modules} \label{sec_TJM_def} \hfill \\
We  elaborate some definitions from Section \ref{introduction}. Let $P=P_{k,n}$ be the  parabolic subgroup of $G_n$ with respect to the partition $(k,n-k)$ of $n$:
 \[P=P_{k,n}=\{\bbsm g_1 & Y \\ & g_2 \besm : g_1 \in G_k, \,g_2 \in G_{n-k}, \,Y \in \M_{k \times n-k}(F) \}. \]
Let $P=MN$ be the  Levi decomposition of $P$, where the unipotent radical $N$ and the Levi subgroup $M$ are:
\[ N=\{ g \in P \colon g_1=I_k, g_2=I_{n-k}\}, \quad M=\{g \in P \colon Y=0\} \simeq G_k \times G_{n-k}.\]
 Since $N \simeq \M_{k \times n-k}(F)$, the   group $\widehat{N}$ of characters of $N$ is 
 \[ \widehat N=\{\psi_A \colon A \in \M_{k \times n-k}(F)\}, \quad \text{where    }  \psi_A(g)=\psi_0(\tr(A^\top Y)).\] We define the rank of the  character $\psi_A$ to be the rank of $A$. We say $\psi_A$ is  non-degenerate if $A$ has full rank, and we say $\psi_A$ is a degenerate character if $A$ does not have full rank.   We recall from Section \ref{introduction}, the subgroup $M_\psi$ of $M$ defined by 
\[M_{\psi} = \{m \in M \colon {\psi}^{m}(n)=\psi(n) , \forall n \in N \},\]
where $\psi^m(n):=\psi(m^{-1}nm)$. We also recall the $M_\psi$-modules
\begin{align*}
    V^{N, \psi} & =\{v \in V \colon \tau(n)v=\psi(n)v, \forall n \in N \}\\
    V(N,\psi) &= \langle \tau(n)v-\psi(n)v \colon n\in N, v\in V \rangle_{\mathbb{C}} \\
    V_{N,\psi}&=V/V(N,\psi). 
\end{align*}
It is easy to see that the linear map Pr$_N:V \to V$ defined by \[ \text{Pr}_N = \frac{1}{|N|} \sum_{n\in N} \tau(n) \overline{\psi(n)},\] is a projection of $V$ onto $V^{N, \psi}$ with kernel $V(N, \psi)$. Therefore, $V=V^{N,\psi} \oplus V(N, \psi)$ as $M_\psi$-modules, and hence $V^{N,\psi} \simeq V_{N,\psi}$. In this work, we will not use the quotient module $V/V(N,\psi)$,  and work exclusively with the submodule $V^{N,\psi}$. By a slight abuse of notation, we will use the symbol $V_{N,\psi}$ to denote the $M_\psi$-module  $V^{N,\psi}$:
\begin{definition} \label{V_N_psi_def}
    Let  $(\tau,V)$ be a  representation of $P\subset G_n$, and let  $\psi$ be a character of $N$. Let
    $M_\psi$ be the subgroup of $M$ defined above, and let $V_{N,\psi}$ be $M_\psi$-module which is the $\psi$-isotypic component of $\tau|_{N}$:
\[V_{N,\psi}= \{v \in V \colon \tau(n)v=\psi(n)v, \forall n \in N \}.\]
 By the \emph{twisted Jacquet module} of $\tau$ with respect to $(N,\psi)$, denoted $(\tau_{N,\psi}, V_{N,\psi})$, is the $M_\psi$-module $V_{N,\psi}$.
\end{definition}

 Let $\chi_{\tau_{N, \psi}}$ denote the character of $\tau_{N, \psi}$. The character $\chi_{\tau_{N, \psi}}$ can be calculated in terms of the character $\chi_\tau$  by the following well known result 
 We include a proof here, for the sake of completeness:
\begin{proposition} \label{character_Jacquet}
Let $(\tau,V)$ be a representation of a parabolic subgroup $P \subset G_n$,  and let $\chi_\tau$ be the character of $\tau$. For $m \in M_\psi$, we have
\[\chi_{\tau_{N, \psi}}(m) = \frac{1}{|N|}\sum_{n \in N}\chi_{\tau}(mn)\overline{\psi(n)}.\]
\end{proposition}
\begin{proof}
The character  $\chi_{\tau_{N, \psi}}(m)$ is the trace of the restriction of $\tau(m)$ to ${V_{N,\psi}}$, which is also equal to the trace of $\tau(m) \text{Pr}_N$.   The latter quantity equals
\[ \tr\left( \tfrac{1}{|N|} \sum_{n\in N} \tau(m)\tau(n) \overline{\psi(n)}\right)=\tfrac{1}{|N|} \sum_{n\in N}  \chi_{\tau}(mn) \overline{\psi(n)}.\]

\end{proof}
For  $g=\bbsm g_1 & \\& g_2 \besm \in M$, the action of $g$ on $N \simeq \M_{k \times n-k}(F)$ is given by $g \cdot X= g_1 X g_2^{-1}$.
The dual action of $M$ on $\widehat N$ is  given by:  
\[g \cdot \psi_A(X)=\psi_A(g^{-1} \cdot X)=\psi_0(\tr(g_2 A^\top g_1^{-1} X))=\psi_{ g_1^{-\top} A g_2^{\top}}(X),\] 
i.e.,  $g \cdot \psi_A=\psi_{ g_1^{-\top} A g_2^{\top}}$.
Clearly, the $M$-orbits on $\widehat N$ are represented by $\psi_r= \psi_{A_r}$ where
\[A_r=\begin{bmatrix}I_r & 0 \\
0 & 0 \end{bmatrix} \in  \M_{k \times n-k}(F), \quad 0 \leq r \leq \min\{k,n-k\}. \]
For $\psi_r$,  the subgroup $M_r=M_{\psi_r}$ of $M$ consists of those $\bbsm g_1& \\ &  g_2 \besm \in M$ such that  $g_1^{-\top} A_r g_2^{\top}=A_r$. Writing $g_1^{-\top} = \bbsm g_{11} & g_{12} \\ g_{13} & g_{14} \besm$  and 
$g_2^{-\top}=\bbsm g_{21} & g_{22} \\ g_{23} & g_{24} \besm$ as block matrices corresponding to the partition $(r,k-r)$ of $k$, and $(r, n-k-r)$ of $n-k$,
we need 
\[ \bbsm g_{11} & 0  \\ g_{13} & 0 \besm =\bbsm g_{11} & g_{12} \\ g_{13} & g_{14} \besm \bbsm 
I_r & 0 \\
0 & 0
\besm = \bbsm 
I_r & 0 \\
0 & 0
\besm
\bbsm g_{21} & g_{22} \\ g_{23} & g_{24} \besm=
\bbsm g_{21} & g_{22} \\ 0 & 0 \besm.
\]
This is  equivalent to $g_{13}=0, g_{22}=0$, and $g_{11}=g_{21}$.  We summarize this:
\begin{definition} \label{def_Mpsi}
    For the parabolic subgroup $P = P_{k,n}$ of $G_n$ with Levi decomposition $P = MN$, and the character $\psi=\psi_r$ of $N$ as above, we denote the group $M_\psi$ as just $M_r$. For $(\tau,V) \in {\rm Rep}(G_n)$ or more generally for $(\tau,V) \in {\rm Rep}(P)$, we denote the twisted Jacquet module $(\tau_{N, \psi}, V_{N,\psi})$ as just $(\tau_r, V_r)$. The group $M_r$ is described by 
    \begin{multline} \label{eq:Mpsi}
    M_r=\Bigg\{ \bbsm h_1 &  0 & 0&0\\
    h_5 & h_4&0&0\\
    0&0&h_1& h_3\\0&0&0&h_2 \besm \colon h_1\in G_r, h_4\in G_{k-r}, h_2\in G_{n-k-r},\\  
    h_5\in \M_{k-r \times r}(F), \, h_3\in \M_{ r \times n-k-r}(F)
    \Bigg\}.
\end{multline}

\end{definition}

\subsection{The twisted  Jacquet module $\pi_{N, \psi_k}$ in the case $n=2k$} \hfill \\
We first record some facts that we need about the characters of the Steinberg representation of $G_n$ (\cite{Carter_finite}) and the representation  
$\Ind_{F_n^\times}^{G_n} \theta$ where $\theta$ is a character of $F_n^\times$. The following result is from  
O. Gorodetsky and Z. Hazan \cite{HazGor}.
\begin{lemma} \cite[eq 6.2]{HazGor}\label{St_char}
    Let $\St_{G_n}$ denote the Steinberg representation of $G_n$. We use the same symbol $\St_{G_n}$ for the character of $\St_{G_n}$. If $g$ is non-semisimple, then $\St_{G_n}(g)=0$. If $g$ is semisimple and comes from $F_n^{\times}$ with minimal polynomial $p(X) \in F[X]$ of degree $m|n$, we have 
    \[\St_{G_n}(g)= (-1)^{n-n'} {q'}^{\tbinom{n'}{2}}, \quad \text{where 
 } n'=n/m, \; q' = q^m.\]
\end{lemma}
The following result is from \cite{Dip[1]}:
\begin{lemma} \cite[Lemma 1]{Dip[1]} \label{ind_char}
For a character $\theta$ of $F_k^\times$, the character $\chi_{\Ind_{F_k^\times}^{G_k} \theta}$ of the induced representation
$\Ind^{G_k}_{F_k^\times} \theta$ equals $0$ unless $g$ is semisimple and comes from $F_k^{\times}$. In case $g$ comes from $F_k^{\times}$ and $F[\lambda]=F_m$ for an eigenvalue $\lambda$ of $g$, we have
\[ \chi_{\Ind_{F_k^\times}^{G_k} \theta}(g) = {q'}^{\tbinom{k'}{2}}\left( \sum_{i=0}^{m-1} \theta(\phi^i(\lambda))\right)  \prod_{j=1}^{k'-1}({q'}^j-1)
\quad \text{ where } k'=k/m, \; q' = q^m.\]
 \end{lemma}

\begin{lemma} \label{ind=st_times_cuspidal} 
Let $\theta$ be a character of $F_n^\times$, and let $\chi_{\Ind_{F_n^\times}^{G_n} \theta}$ denote the character of the  representation $\Ind_{F_n^\times}^{G_n} \theta$. Let  $\St_{G_n}$ be the Steinberg representation   of $G_n$. Let $\Theta_\theta$ be the class function as in Definition \ref{Theta_definition}. \begin{enumerate}
\item[\emph{(1)}] $\chi_{\Ind_{F_n^\times}^{G_n} \theta}(g) = \St_{G_n}(g) \cdot \Theta_\theta(g)$.
\item[\emph{(2)}]  If $\theta$ is a regular character of $F_n^\times$ and $\pi_\theta$, the associated cuspidal representation of $G_n$, then
\[ \Ind_{F_n^\times}^{G_n} \theta  \simeq \St_{G_n} \otimes\,  \pi_\theta.\]
\end{enumerate}
\end{lemma}
\begin{proof}
Part (2) clearly follows from part (1). It suffices to prove $\chi_{\Ind_{F_n^\times}^{G_n} \theta}(g)=\St_{G_n}(g) \cdot \Theta_\theta(g)$. We can assume $g$ is semisimple and comes from $F_n^\times$, for otherwise both  $\Theta_{\theta}(g) \St_{G_n}(g)$ and $\chi_{\Ind_{F_n^\times}^{G_n}}(g)$ are  zero.
In the case when  $g$ comes from $F_n^\times$, it follows from Lemma \ref{St_char} and Lemma \ref{character value of cuspidal representation (Dipendra)} that  $\Theta_\theta(g) \St_{G_n}(g)$  equals the expression for $\chi_{\Ind_{F_n^\times}^{G_n}}(g)$ given in Lemma \ref{ind_char}.
 \end{proof}

In the case when $n=2k$ and $P=MN$ is the maximal parabolic subgroup of $G_n$ corresponding to the partition $(k,k)$ of $n=2k$, the group $M \simeq G_k \times G_k$ and $N$ is isomorphic to the additive group $\M_{k\times k}(F)$. In the case when the character $\psi_k=\psi_A$ of $N$ is non-degenerate (that is  $\Rank(A)=k$), it follows from \eqref{eq:Mpsi} that  the group \[ M_{\psi_k}=\left\{ \bbsm h_1 & \\ & h_1 \besm \colon h_1 \in G_k\right\} \simeq G_k.\]
The  twisted Jacquet module $\pi_{N, \psi_k}$ was determined by  D. Prasad:
\begin{theorem} \cite{Dip[1]}  \label{n=2k_nondeg} For  a cuspidal representation $\pi=\pi_\theta$ of $G_{2k}$ associated to a regular character $\theta$  of $F_{2k}^\times$, we have  
\[ \pi_{N,\psi_k} \simeq \Ind_{F_k^\times}^{G_k} (\theta |_{F_k^\times})\]
as $M_{\psi_k}\simeq G_k$-modules.
\end{theorem}

In the work \cite{HazGor},  O. Gorodetsky and Z. Hazan studied the following generalization of Theorem \ref{n=2k_nondeg} to $G_{kn}$. Let $P=MN$ be the Levi decomposition of the standard  parabolic subgroup $P$ of $G$ corresponding to the partition of $kn$ into $k$ parts each of size $n$. The Levi subgroup $M$ is isomorphic to $G_n\times \dots \times G_n$, and the unipotent radical $N$ consists of block upper triangular matrices where each block has size $n \times n$ and the diagonal blocks are all $I_n$. The non-degenerate characters of $N$ form a single orbit under the action of $M$ on the characters of $N$. For a cuspidal representation $\pi_\theta$ of $G_{kn}$ associated to a regular character $\theta$ of $F_{kn}^\times$, and  a non-degenerate character $\psi$ of $N$, the structure of the module $\pi_{N,\psi}$ is:
\begin{theorem} \cite[Theorem 
 1]{HazGor} For $k \geq 2$, 
let $\pi_\theta$ be 
  a cuspidal representation  of $G_{kn}$ associated to a regular character $\theta$ of $F_{kn}^\times$ and a non-degenerate character $\psi$ of the unipotent radical $N$ of $P$ as above. 
  In this case $M_{\psi} \simeq G_n$ and 
\[\pi_{N,\psi} \simeq \St_{G_n}^{\otimes (k-2)} \otimes  \Ind_{F_n^\times}^{G_n} (\theta |_{F_n^\times}) \]
as an $M_{\psi}$-module. 
\end{theorem}
The version of Theorem 1  as it appears in \cite{HazGor} is that 
\[ \pi_{N,\psi} \simeq \St_{G_n}^{\otimes (k-1)} \otimes  
 \, \pi_{\theta  |_{F_n^{\times}}},\]
which assumes the condition that the restriction $\theta |_{F_n^\times}$ is regular, so that the cuspidal representation $\pi_{\theta  |_{F_n^\times}}$of $G_n$ is defined. The version we stated above is implicit in the results of \cite{HazGor} (Theorem 3, Equation 6.2, and Lemma 2.3), and has been confirmed to be correct  in private communication  with O. Gorodetsky and Z. Hazan. We note that in the special case when $\theta |_{F_n^\times}$ is regular, both the  versions agree, as can be seen  from Lemma \ref{ind=st_times_cuspidal}.

 \section{Some combinatorial  results}\label{combin}
We first record some results that we will need from \cite{LAA}. Let 
\[\M_{m \times n}(r,F)=\{X \in \M_{m \times n}(F) \colon \Rank(X)=r\}, \]
and  let $a(m\times n, r, q)$ denote the cardinality of $\M_{m\times n}(r, F)$. It is not hard to see that 
\[a(m\times n, r, q)=\stirling{m}{r}_q  \stirling{n}{r}_q |G_r|,\] where $\stirling{m}{r}_q$ is the Gaussian binomial coefficient (see \eqref{binomial coefficient}).
For $\alpha\in F$, let 
\begin{equation}
f^\alpha_{m \times n, r, k}(q)=\#\{ X\in \M_{m \times n}(r,F) \colon X_{11}+ \dots+X_{kk}=\alpha\}. 
\end{equation}
\begin{theorem} \cite[Theorem 4.2]{LAA} \label{LAA_thm}
Let $\psi_0$ be a non-trivial additive character of $F$, and let $A \in \M_{m \times n}(k, F)$. The sum 
\[ \sum_{X \in \M_{m \times n}(r,F)}  \psi_0(\tr(A^\top X)),\]
as a function of $A$, only depends on $k=\Rank(A)$ and equals the difference
\[ g_{m \times n, r, k}(q):=f^0_{m \times n, r, k}(q) - f^1_{m \times n, r,k}(q).\] We have:
\begin{equation} \label{eq:gmnrk}
    g_{m \times n,r,k}(q) = \sum_{i=0}^r (-1)^i \stirling{k}{i}_q \,  q^{\tbinom{i}{2} + k(r-i)} \, a( (m-k)\times (n-k),r-i,q).
    \end{equation}
\end{theorem}
We recall the definitions of the $q$-binomial coefficient, the $q$-Pochhammer symbol, and the $q$-binomial theorem.

\begin{definition} Let $n$ be a non-negative integer. Let $(T;q)_n$ be the  polynomial in $\bZ[T]$ of degree $n$  defined by
\[(T;q)_{n}= \prod\limits_{i=0}^{n-1}(1-Tq^{i}),\quad  n>0. \]
We take $(T;q)_0=1$.
The Gaussian or $q$-binomial coefficient 
 $\stirling{n}{r}_{q}$ counts the number of $r$ dimensional vector subspaces  of $F^n$ and is given by the formula
 \begin{equation}\label{binomial coefficient} \stirling{n}{r}_{q}= \frac{(q^{n-r+1};q)_r}{(q;q)_r}. 
 \end{equation}
The following polynomial identity in $\bZ[T]$ is known as the $q$-binomial theorem:
\begin{equation} \label{eq:qbinthm} (T;q)_{n}= \sum_{r=0}^{n}\stirling{n}{r}_{q}(-1)^{r}q^{\tbinom{r}{2}}T^{r}.\end{equation}
\end{definition}
Consider the $\bQ$-vector space  consisting of polynomials in $\bQ[T]$ of degree at most $n$. The polynomials 
$(1,T,T^2, \dots, T^n)$ and $((T;q)_0, \dots, (T;q)_n)$
are two  bases for this vector space. 
The expansion of the second basis vectors in terms of the first basis is given by the $q$-binomial theorem \eqref{eq:qbinthm}. The next result gives the expansion of the first basis vectors in terms of the second basis:
  \begin{lemma} \label{id_T}
\begin{equation} \label{eq:id_T}
q^{\binom{n}{2}} T^n = \sum_{i=0}^n  (-1)^i q^{ \binom{n-i}{2}} \stirling{n}{i}_q  \, (T;q)_i
\end{equation}
\end{lemma}
 \begin{proof}   The RHS of \eqref{eq:id_T}  can be written using the $q$-binomial theorem mentioned above as
\[ \sum_{i=0}^n  (-1)^i q^{ \binom{n-i}{2}} \stirling{n}{i}_q  \, ( \sum_{k=0}^i q^{\binom{k}{2}}\stirling{i}{k}_q  (-T)^k)
\]
Interchanging the order of summation, and using $\stirling{n}{i}_q \stirling{i}{k}_q = \stirling{n}{k}_q \stirling{n-k}{n-i}_q$, we get
\[ (-1)^n \sum_{k=0}^n q^{\binom{k}{2}}  (-T)^{k} \stirling{n}{k}_q     \, ( \sum_{i \geq k} (-1)^{n-i} q^{ \binom{n-i}{2}} \stirling{n-k}{n-i}_q ).
\]
The expression $\displaystyle \sum_{i \geq k} (-1)^{n-i}  q^{ \binom{n-i}{2}} \stirling{n-k}{n-i}_q$ equals $(1;q)_{n-k} = \delta_{k,n}$. Thus, the above expression is just $q^{\binom{n}{2}}  T^n$.
\end{proof} 

\begin{lemma} \label{amnr_lemma}
The following polynomial identity holds in the polynomial ring  $\bQ[S]$:
\[\sum_r a(m \times n,r, q) (S;q)_{n-r}= q^{mn} (Sq^{-m};q)_n.\] 
\end{lemma}
\begin{proof}
For $\ell \geq n$, the size of the set $\M_{n \times \ell}(n,F)$ of all full rank $n \times \ell $ matrices is 
 \[ a( n\times \ell, n,q)=\prod_{i=0}^{n-1} (q^\ell-q^i).\]
 For such a matrix $X \in \M_{n \times \ell}(n,F)$, and for $m \leq \ell$,  let $X=\begin{bmatrix}  X_1 &  X_2 \end{bmatrix}$ where $X_1$ has size $n \times m$ and $X_2$ has size $ n \times (\ell-m)$. 
For $X_1$ of rank $r$, let $U$ denote the column span of $X_1$. The number of matrices $X_2$ such that 
$\begin{bmatrix}  X_1 &  X_2 \end{bmatrix}$ has rank $n$, is same as the number of maps $F^{\ell-m} \to F^n$ such that the associated map $F^{\ell-m} \to F^n/U$ has rank $(n-r)$. In other words, the number of choices for $X_2$ is \[ q^{(\ell-m)r} \cdot a(n-r \times \ell-m, n-r, q)= q^{(\ell-m)r} \cdot \prod_{i=0}^{n-r-1}(q^{\ell-m}-q^i). \]

 Therefore, we get the identity
\[  \prod_{i=0}^{n-1} (q^\ell-q^i)=\sum_{r} a(m \times n, r, q) q^{(\ell-m)r} \prod_{i=0}^{n-r-1}(q^{\ell-m}-q^i). \]
Multiplying 
 by $q^{mn}$ we can rewrite this as 
 \[  q^{mn} \prod_{i=0}^{n-1} (q^\ell-q^i)=\sum_{r} a(m \times n, r, q) q^{\ell r} \prod_{i=0}^{n-r-1}(q^{\ell}-q^{i+m}). \]
This implies the polynomial identity in $\bQ[T]$:
\[
q^{mn} \prod_{i=0}^{n-1} (T-q^i)=
\sum_{r} a(m \times n, r, q) T^{r} \prod_{i=0}^{n-r-1}(T-q^{i+m}),
\]
because the difference between the RHS and LHS is a polynomial of degree at most $n$ in $\bQ[T]$ with infinitely many roots $T=q^\ell$ for $\ell \geq n$, and hence it is the  zero polynomial. Dividing by $T^n$, we get the identity in $\bQ[T^{-1}]$:
\[
q^{mn} \prod_{i=0}^{n-1} (1 -q^i T^{-1})=
\sum_{r} a(m \times n, r, q)   \prod_{i=0}^{n-r-1}(1-q^{i+m}T^{-1}).
\]
In terms of $S=q^{m}T^{-1}$ we can rewrite this as the identity in $\bQ[S]$ given by:
\[
q^{mn}(Sq^{-m};q)_n=
\sum_{r} a(m \times n, r, q) \, 
(S;q)_{n-r}.
\]
\end{proof}

\section{A fundamental  representation of $P_{k,n}$.} \label{section_Pi}
\subsection{Generalizing the Bernstein-Zelevinsky functors to Parabolic subgroups of $G_n$}
\label{sec_equivalence} 
We begin by recalling and fixing some notation. For a representation $(\tau,V)$ of $P = P_{k,n}$, and for $0 \leq r \leq \mu:=\min\{k,n-k\}$, we recall from Definition \ref{def_Mpsi}, that the rank $r$ twisted Jacquet module $(\tau_r,V_r)$ of $\tau$ is a representation of the group $M_r \subset P^*_{r,k} \times P_{r,n-k}$ given by 
\[
    M_r=\left\{\bbsm g_1 &  \\
 &g_2 \besm \colon 
    g_1=\bbsm h_1 &  0\\h_5 & h_4 \besm \in P^*_{r,k}, \;  g_2=\bbsm h_1& h_3\\0&h_2 \besm \in P_{r,n-k} \right\}. \]
We note that the same  matrix  $h_1 \in G_r$ occurs in both the blocks $g_1$ and $g_2$.\\

Given two representations $(\tau, V)$ and $(\tau',V')$ of $P$, and an intertwining operator $f:(\tau,V) \to (\tau',V')$, we have a corresponding intertwining operator $f_r:(\tau_r,V_r) \to (\tau_r', V_r')$.\\

Let $P(r)=N \rtimes M_r$ be the subgroup of $P$ consisting of matrices $\bbsm g_1 & g_1g_3\\ & g_2 \besm$ such that $\bbsm g_1 & \\ & g_2 \besm \in M_r$. The character $\psi_r$ of $N$ extends to a well-defined character of $P(r)$ by $\psi_r(mn)= \psi_r(n)$, which we denote by the same symbol $\psi_r$.\\

Any representation $\sigma$ of $M_r$ inflates to a representation of $P(r)$ via the epimorphism $P(r) \to M_r$ defined by $\bbsm g_1 & g_1g_3\\ & g_2 \besm \mapsto \bbsm g_1 & \\ & g_2 \besm$. We continue to denote the inflated representation of $P(r)$ by $\sigma$.
\begin{definition} \label{Phi_r_def} \hfill  \begin{enumerate}
\item The functor $\Phi_r^- \colon {\rm Rep}(P) \to {\rm Rep}(M_r)$ is defined on objects by 
\[ \Phi_r^-(\tau,V)=(\tau_r,V_r),\] and on morphisms $f:(\tau,V) \to (\tau',V')$ by $\Phi_r^-(f)=f_r$. \\

\item The functor $\Phi_r^+\colon  {\rm Rep}(M_r) \to {\rm Rep}(P)$ is defined in terms of  the induction functor by \[ \Phi_r^+(\sigma)=\Ind_{P(r)}^P (\psi_r \otimes \sigma).\]
\item Viewing ${\rm Rep}(M_r)$ and $ {\rm Rep}(P)$ as Abelian categories, we consider the twisted Jacquet functor 
\[ \fF^- \colon {\rm Rep}(P) \to \bigoplus_{r=0}^\mu {\rm Rep}(M_r), \qquad \fF^-(\tau)=(\Phi_0^-(\tau), \dots, \Phi_{\mu}^-(\tau)),  \]
and the functor
\[ \fF^+ \colon  \bigoplus_{r=0}^\mu {\rm Rep}(M_r) \to {\rm Rep}(P), \qquad \fF^+(\sigma_0, \dots, \sigma_\mu) = \oplus_{r=0}^\mu \Phi^+_r(\sigma_r).  \]
\end{enumerate}
\end{definition}
In the case $k=n-1$, we have $P=F^\times \times \Mir_n$, where 
$F^\times$ is the subgroup of scalar matrices, and $\Mir_n$ is the Mirabolic subgroup of $P$ consisting of matrices in $P$ whose $(n,n)$-th entry is $1$. Here $\mu=\text{min}\{n-1,1\}=1$, and  $M_0 \simeq F^\times  \times G_{n-1}$ and $M_1 \simeq F^\times  \times \Mir_{n-1}$. The functors $\Phi_0^{\pm}$ and $\Phi^{\pm}_1$ are of the form \[\Phi_0^-=1 \boxtimes \Psi^-, \quad \Phi_1^-=1 \boxtimes \Phi^-, \quad  
\Phi_0^+=1 \boxtimes \Psi^+, \quad \Phi_1^+=1 \boxtimes \Phi^+,\]
where $ {\rm Rep}(\Mir_{n-1}) \xrightleftharpoons
[\Phi^-]{\Phi^+} {\rm Rep}(\Mir_n)
\xrightleftharpoons
[\Psi^+]{\Psi^-} {\rm Rep}(G_{n-1})$
are the Bernstein-Zelevinsky functors  
(see \cite[Section 5.11]{BerZel[2]}).
The properties of the functors $\Psi^{\pm}$ and $\Phi^{\pm}$ given in \cite[Proposition 5.12]{BerZel[2]} generalize to the functors $\Phi_r^{\pm}$. The next theorem generalizes the parts (a)-(f) of that Proposition, whereas part (g) is treated in Section \ref{sec_Pi_construction}. The representation $\tau^0$ of part (g) of the Proposition is generalized to the representation $\Pi_{n-k,n}^\dagger$ of $P_{k,n}$ that we construct in Section \ref{sec_Pi_construction}.
\begin{theorem} [Properties of the functors $\Phi_r^{\pm}$] \label{Phi+-_properties}  \hfill
\begin{enumerate} [label=\alph*)]
\item[\emph{(a)}] The functors $\Phi_r^{\pm}$ are exact. 
\item[\emph{(b)}] The pair of functors $\Phi_r^-$ and $\Phi_r^+$ are (left and right) adjoints of each other.
\item[\emph{(c)}]  For $r \neq s$, we have $\Phi_r^- \Phi_s^+=0$. For $\sigma \in {\rm Rep}(M_r)$, we have $\Phi_r^- \Phi_r^+(\sigma) \simeq \sigma$.
\item[\emph{(d)}] For $\tau \in {\rm Rep}(P)$, we have 
        \[ \tau \simeq \oplus_{r=0}^\mu \Phi_r^+ \Phi_r^-(\tau).\]
In particular, the pair of functors 
\[ {\rm Rep}(P) \xrightleftharpoons[\fF^+]{\fF^-}\bigoplus_{r=0}^\mu {\rm Rep}(M_r)\]
give an equivalence between the categories ${\rm Rep}(P)$ and $\oplus_{r=0}^\mu {\rm Rep}(M_r)$.\\
\item[\emph{(e)}] For $\tau \in {\rm Rep}(P)$,  the condition that $\tau|_N$ is identity, is equivalent to the condition that $\Phi_r^-(\tau)=0$ for all $r \geq 1$.\\
\item[\emph{(f)}] If $\sigma \in {\rm Rep}(M_r)$, then $\Phi_r^+$  and $\Phi_r^-$  establish a bijection between $\sigma$ and $\Phi_r^+(\sigma)$. In particular, $\sigma$  and $\Phi_r^+(\sigma)$  are irreducible simultaneously. 
\end{enumerate}
\end{theorem}
\begin{proof}
(a) The exactness of the functors $\Phi_r^+$ follows from the exactness of the induction functor. As for the exactness of $\Phi_r^-$, applying this  functor to a given exact sequence $0 \to (\tau',V') \xrightarrow{\imath} (\tau, V) \xrightarrow{\jmath} (\tau'', V'') \to 0$, we get maps $(\tau_r',V_r') \xrightarrow{\imath_r} (\tau_r, V_r)$ and $ (\tau_r, V_r) \xrightarrow{\jmath_r} (\tau_r'', V''_r)$. Since $V'_r, V_r, V_r''$ are submodules of $V', V, V''$ for us,  clearly (i)  $V_r''=\jmath(V_r)$, (ii) $\ker(\imath_r)=0$, and (iii) $\ker(\jmath_r)=\ker(\jmath) \cap V_r = \imath(V') \cap V_r = \imath_r(V'_r)$, which shows that $0 \to (\tau_r',V_r') \xrightarrow{\imath_r} (\tau_r, V_r) \xrightarrow{\jmath_r} (\tau_r'', V_r'') \to 0$ is an exact sequence.\\

(b) For $\sigma \in {\rm Rep}(M_r)$ and $\tau \in {\rm Rep}(P)$, we have by Frobenius reciprocity and Proposition \ref{character_Jacquet} that: 
\[ \langle \chi_{\Phi_r^+(\sigma)}, \chi_\tau \rangle_P = \langle \chi_\sigma, \bar{\psi_r} \otimes \chi_{\tau|_{P(r)}} \rangle_{P(r)}=\langle \chi_\sigma, \chi_{\Phi_r^-(\tau)} \rangle_{M_r}.\]
This shows that $\Phi_r^+$ and $\Phi_r^-$ are a pair of (2-sided) adjoint functors.\\

(c)  We recall that by the \emph{little group method} for representations of a group with an abelian normal subgroup  (for example \cite[Theorem 5.2]{CecScaTol}),  any irreducible representation of $P$ is of the form $\Ind_{P(r)}^P (\psi_r \otimes \sigma)=\Phi_r^+(\sigma)$  for some $0 \leq r \leq \mu$ and an irreducible representation $\sigma$ of $M_r$. In particular, for irreducible representations $\sigma$ of $M_r$ and $\eta$ of $M_s$, we have $\langle \chi_{\Phi_r^+(\sigma)}, \chi_{\Phi_s^+(\eta)} \rangle_P$ is $0$ unless $r=s$ and $\eta \simeq \sigma$. In particular, for $r \neq s$ we have:
\[ 0= \langle \chi_{\Phi_r^+(\sigma)}, \chi_{\Phi_s^+(\eta)} \rangle_P \overset{\text{by part b)}}{=} \langle \chi_\sigma, \chi_{\Phi_r^- \Phi_s^+(\eta)} \rangle_{M_r}.\]
Since this is true for every irreducible $\sigma$, we conclude $\Phi_r^{-} \Phi_s^+ =0$.\\
On the other hand if $r=s$, we get
\[  \langle \chi_\sigma, \chi_{\Phi_r^- \Phi_r^+(\eta)} \rangle_{M_r} \overset{\text{by part b)}}{=}  \langle \chi_{\Phi_r^+(\sigma)}, \chi_{\Phi_r^+(\eta)} \rangle_P   = \begin{cases}
    1 &\text{ if  $\eta \simeq \sigma$}\\0  &\text{ if  $\eta \not\simeq \sigma$}
\end{cases}.\]
Thus we have shown $\Phi_r^- \Phi_r^+(\sigma) \simeq \sigma$ for every irreducible $\sigma$ and hence for any $\sigma \in {\rm Rep}(M_r)$.\\

(d) We recall that the symbol $\widehat{\mG}$ denotes the set of isomorphism classes of irreducible representations of a finite group $\mG$. As observed above, any $\tau \in \widehat{P}$  is isomorphic to $\Phi_r^+(\sigma_r)$ for some $0 \leq r \leq \mu$ and some $\sigma_r \in \widehat{M_r}$. Therefore, given $\tau \in  {\rm Rep}(P)$ there exist non-negative integers $a_{r, \sigma_r}$ such that 
\[\tau \simeq \bigoplus_r  \bigoplus_{ \sigma_r \in \widehat{M_r}} a_{r, \sigma_r} \Phi_r^+(\sigma_r).\]
By part (c), we get 
\[ \Phi_r^-(\tau)= \bigoplus_{ \sigma_r \in \widehat{M_r}} a_{r, \sigma_r} \, \sigma_r.\]
Therefore, $\Phi_r^+ \Phi_r^-(\tau) \simeq \oplus_{\sigma_r \in \widehat{M_r}} a_{r,\sigma_r} \Phi_r^+(\sigma_r)$, and hence  $\tau \simeq  \oplus_r \Phi_r^+ \Phi_r^-(\tau)$.\\

In terms of the functors $\fF^{\pm}$, we have shown $\fF^+ \circ \fF^-$ is isomorphic to the identity functor of 
${\rm Rep}(P)$. Similarly, given $(\eta_0, \dots, \eta_\mu) \in \oplus_r {\rm Rep}(M_r)$, we  have  $\fF^+(\eta_0, \dots, \eta_\mu)=\oplus_r  \Phi_r^+(\eta_r)$. Applying the functor $\fF^-$ to this, and using part c), we get $\fF^- \fF^+(\eta_0, \dots, \eta_\mu) \simeq   (\eta_0, \dots, \eta_\mu)$, which shows that $\fF^- \circ \fF^+$ is isomorphic to the identity functor of $\oplus_r {\rm Rep}(M_r)$. Thus, we have shown that the functors $\fF^{\pm}$ establish an equivalence between the categories  $ {\rm Rep}(P)$ and $\oplus_r {\rm Rep}(M_r)$.\\

(e) For $(\tau,V) \in {\rm Rep}(P)$, the condition that $\tau(n)$ is the identity transformation of $V$ for all $n \in N$, is clearly equivalent to the condition that $V_{N, \psi} =\{0_V\}$ for a non-trivial character $\psi$ of $N$, and $V_{N, \psi}=V$ for the trivial character $\psi$ of $N$.\\

(f) If $(\tau,V)= \Phi_r^+(\sigma)$, then for any subrepresentation $(\tau', W)$ (satisfying $\tau'(g)=\tau(g)|_W$ for $g \in P$), we have by part (d) that $\tau(g)|_W =(\Phi_r^+ \Phi_r^- \tau)(g)|_W$, and hence $\tau'= \Phi_r^+ \Phi_r^-(\tau')$.
\end{proof}

\begin{corollary} \label{Cat_equiv_corollary}
     A representation $\tau$ of the maximal parabolic subgroup $P=P_{k,n}$ of $G_n$ is irreducible if and only if all but one of the  twisted Jacquet modules $\tau_0, \dots, \tau_\mu$ are zero, and the non-zero one is irreducible.
 \end{corollary}
 \begin{proof}
This is clear from the  equivalence between the categories   $\oplus_{r=0}^\mu {\rm Rep}(M_r)$  and ${\rm Rep}(P)$.
\end{proof}
We  illustrate Corollary \ref{Cat_equiv_corollary}
for $P_{n-1,n} =F^\times \times \Mir_n$. It suffices to consider representations $\tau = 1 \boxtimes \pi$ for $\pi \in{\rm Rep}(\Mir_n)$. Here $\Phi_0^-(\tau)=1 \boxtimes \Psi^-(\pi)$ and $\Phi_1^-(\tau)=1 \boxtimes \Phi^-(\pi)$. The corollary asserts that $\tau$ is  irreducible if either (i) $\pi=\Psi^+(\pi_0)$ for some $ \pi_0\in \widehat{G_{n-1}}$ or (ii) $\pi=\Phi^+(\pi_1)$ for some $ \pi_1\in \widehat{\Mir_{n-1}}$. In case (ii), we can repeat this procedure  to conclude that $\pi$ must be isomorphic to $(\Phi^+)^{r-1} \Psi^+(\pi_r)$ for $\pi_r \in \widehat{G_{n-r}}$, which is \cite[Corollary 5.13]{BerZel[2]}.

\subsection{Construction of $\Pi_{k,n}$} \label{sec_Pi_construction}
Continuing with the last paragraph, the representation $\tau^0:=(\Phi^+)^{n-1} \Psi^+(1) = (\Phi^+)^{n-1}(1)$ of $\Mir_n$ is irreducible, and by part (f) of Proposition 5.12 of \cite{BerZel[2]}, it is isomorphic to $\Ind_{U_n}^{\Mir_n} \vartheta$, where $U_n$ is the  unipotent subgroup of $\Mir_n$ consisting of upper triangular matrices with ones on the diagonal, and $\vartheta$ is the character of $U_n$ defined by $\vartheta(g)=\tilde\psi_0(g_{1,2}+ \dots+g_{n-1,n})$ with $\tilde\psi_0$ a non-trivial additive character of $F$. By \cite[5.18]{BerZel[2]}, the restriction to $\Mir_n$ of any cuspidal representation of $G_n$ is isomorphic to $(\Phi^+)^{n-1}(1)$.
It is to be noted that the symbol $(\Phi^+)^{n-1}$ actually stands for 
\[ {\rm Rep}(\Mir_1) \xrightarrow{\Phi^+} {\rm Rep}(\Mir_2) \xrightarrow{\Phi^+} \cdots \xrightarrow{\Phi^+} {\rm Rep}(\Mir_{n-1}) \xrightarrow{\Phi^+} {\rm Rep}(\Mir_n),\]
and hence $(\Phi^+)^{n-1}(1)$ is a  recursively defined representation.
In this section, we  construct a representation $\Pi_{k,n}$ (and a closely related representation $\Pi_{n-k,n}^\dagger$) of the parabolic group $P_{k,n}$ for $1 \leq k \leq n$, which will satisfy a universal property with respect to restrictions to $P_{k,n}$ of cuspidal representations of $G_n$ (see Theorem \ref{Pi_universal_property}).
For $k=n-1$, the representation $\Pi_{n-k, n}^\dagger$ of $P_{n-1,1}=F^\times \times \Mir_n$ will be shown to be isomorphic to  $1 \boxtimes (\Phi^{+})^{n-1}(1)$.\\

We take the  representation $\Pi_{n,n}$ of $P_{n,n} =G_n$ to be the trivial representation. For $1 \leq k<n$, using the categorical equivalence  ${\rm Rep}(P) \xrightleftharpoons[\fF^+]{\fF^-}\bigoplus_{r=0}^\mu {\rm Rep}(M_r)$ 
given in  part (d) of Theorem \ref{Phi+-_properties}, 
it suffices to define $\fF^-(\Pi_{k,n})$, i.e. the twisted Jacquet modules $\Phi_r^-(\Pi_{k,n})$.  For $g=\bbsm g_1&0\\0&g_2 \besm \in M_r$ with 
$g_1=\bbsm h_1 &  0\\
    h_5 & h_4\besm$ and $g_2=\bbsm
    h_1& h_3\\0&h_2 \besm$, let 
    \[  M_r \xrightarrow{p_2} P_{r,n-k}, \quad  M_r \xrightarrow{p_3} G_r,\]
    be the epimorphisms defined by 
    \[  p_2(g)=g_2, \quad  p_3(g)=h_1.\]
Any representation of $P_{r,n-k}$ can be inflated to a representation of $M_r$ via $p_2$. Let $\St_{G_r}$ be the representation of $P_{r,n-k}$ obtained by inflating the Steinberg representation  of $G_r$ via $p_3$. We recursively define
\[ \Phi_r^-(\Pi_{k,n})= \begin{cases}
     0 &\text{ if $r=0$}\\
\St_{G_r} \otimes \Pi_{r,n-k} &\text{ if $r>0$}.     
\end{cases}\]
Here $\St_{G_r}$ and  $\Pi_{r,n-k}$ are inflated to $M_r$ via the epimorphisms $p_3$ and $p_2$ respectively. Having defined $\fF^-(\Pi_{k,n})$, we  obtain $\Pi_{k,n}$ as $\fF^+ \fF^-(\Pi_{k,n})$:
\begin{definition} \label{def_Pi} The representation $\Pi_{k,n}$ of $P_{k,n}$ for $1 \leq k \leq n$ is the trivial representation of $P_{n,n}$ in case $k=n$, and in case $1 \leq k \leq n-1$, it is defined by
\[\Pi_{k,n}=\bigoplus_{r=1}^{\text{min}\{k,n-k\}} \Ind_{P_{k,n}(r)}^{P_{k,n}} ( \psi_r \otimes [\St_{G_r} \otimes \Pi_{r,n-k}])
    \]
    where $P_{k,n}(r)=M_r \ltimes N_{k,n}, N_{k,n} \simeq \M_{k \times n-k}(F).$
\end{definition}
The recursion terminates in exactly  $(n-k)$  steps. To see this, we induct on the `length' len$(\Pi_{k,n}):=n-k$. If len$(\Pi_{k,n})=0$ i.e. $n=k$, we have $\Pi_{n,n}=1_{P_{n,n}}$ and we do not use the recursion.  For $\Pi_{k,n}$ with $(n-k)=\ell$, we assume inductively that that the recursion for $\Pi_{i,j}$ terminates in $(j-i)$ steps if $(j-i)<\ell$.  It follows that  the representations $\Pi_{r,n-k}$ appearing in the first  recursion for $\Pi_{k,n}$, in turn  require $(n-k-r)$ steps, because $r \geq 1$ and hence $n-k-r=\ell-r <\ell$. In particular, the recursion for $\Pi_{1,n-k}$ requires $\ell-1=(n-k-1)$ steps. Therefore, we conclude that the recursion for $\Pi_{k,n}$ terminates in $(n-k)$ steps. \\
 
 The  simplest example of $\Pi_{k,n}$ after $\Pi_{n,n}$, is  $\Pi_{n-1,n}$. We use the functor $\Phi^+ : {\rm Rep}(\Mir_{n-1}) \to {\rm Rep}(\Mir_n)$ (see the discussion after Definition \ref{Phi_r_def}). 
\begin{lemma} \label{Pi_n-1_n}
The representation $\Pi_{n-1,n}$ of $P_{n-1,n}=F^\times \times \Mir_n$ has dimension $(q^{n-1}-1)$ and  equals $1_{F^\times} \boxtimes \Phi^+(1_{\Mir_{n-1}})$. 
\end{lemma}
\begin{proof}
By definition \ref{def_Pi}, we have $\Pi_{n-1,n}=\Ind_{P_{n-1,n}(1)}^{P_{n-1,n}} (\psi_1 \otimes [\St_{G_1} \otimes \Pi_{1,1}])$. Since $\St_{G_1} = 1_{G_1}$ and $\Pi_{1,1}=1_{P_{1,1}}$ are trivial representations, it follows that $[\St_{G_1} \otimes \Pi_{1,1}]$, viewed as a representation of $P_{n-1,n}(1)$,  is the trivial representation.
We recall that  $P_{n-1,n}=F^\times \times \Mir_n$ and $P_{n-1,n}(1)=F^\times \times (\Mir_{n-1} \ltimes F^{n-1})$, where $F^\times$ is the group of scalar matrices in $G_n$, and \begin{align*}
    \Mir_n&=\{ \bbsm h & v \\ & 1 \besm \in P_{n-1,n} \colon h \in G_{n-1}, v \in F^{n-1}\}, \\
\Mir_{n-1} \ltimes F^{n-1} &= \{ \bbsm h & v\\ & 1 \besm \in \Mir_n \colon h \in \Mir_{n-1}\}.
\end{align*}
Therefore,
\[ \Pi_{n-1,n}=1_{F^\times} \boxtimes \Ind_{\Mir_{n-1} \ltimes F^{n-1}}^{\Mir_n} (\psi_1 \otimes 1_{\Mir_{n-1}}). \]
In particular, 
\[ \dim(\Pi_{n-1,n})=|\Mir_n|/|\Mir_{n-1} \ltimes F^{n-1}|=(q^{n-1}-1).\]
For a representation $\tau$ of $\Mir_{n-1}$, we recall that 
$\Phi^+(\tau)$ is the representation of $\Mir_n$ given by $\Ind_{\Mir_{n-1} \ltimes F^{n-1}}^{\Mir_n} (\psi_1 \otimes \tau)$. Thus, we conclude that $\Pi_{n-1,n}=1_{F^\times} \boxtimes \Phi^+(1_{\Mir_{n-1}})$.
 \end{proof}

The representation $\Pi_{1,n}$ of $P_{1,n}$ can also be described in terms of the $\Phi^+$ functor. We start with  a definition:
\begin{definition}
\label{dagger_def}
Let $\dagger: G_n \to G_n$ be the involutive automorphism  given by 
\[ g^\dagger = \bbsm &&&1 \\
&&1&\\
&\iddots&&\\
1&&&
\besm g^{-\top} \bbsm &&&1 \\
&&1&\\
&\iddots&&\\
1&&&
\besm. \]
\end{definition}
For a subgroup $H$ of $G_n$, we denote $\dagger(H)$ by $H^\dagger$, and for a representation $\pi$ of $H$, the representation $\pi \circ \dagger$ of $H^\dagger$ will be denoted $\pi^\dagger$. We note that $P_{k,n}^\dagger = P_{n-k,n}$ for $1 \leq k \leq n-1$.
For a cuspidal representation $\pi_\theta$ of $G_n$, we claim that $\pi_\theta^\dagger$ is isomorphic to $\pi_{\bar \theta}$
 where $\bar \theta$ is the complex conjugate of the regular character $\theta$ of $F_n^\times$, i.e., 
\begin{equation} \label{eq:pi_theta_dagger}
\pi_\theta^\dagger \simeq \pi_{\bar \theta}.
\end{equation} 
To see this, we note that the character  $\chi_{\pi_\theta^\dagger}(g)=\chi_{\pi_\theta}(g^{-\top})=\Theta_{\theta}(g^{-\top})$ because $g^\dagger$ is conjugate to $g^{-\top}$ by definition. Since $g^\top$ is conjugate to $g$ for all $g \in \GL(n,K)$ (for an arbitrary field $K$), it follows that $\chi_{\pi_\theta^\dagger}(g)=
\Theta_{\theta}(g^{-1})$.
It follows from Theorem \ref{character value of cuspidal representation (Dipendra)}, that $\Theta_{\theta}(g^{-1})=\Theta_{\bar \theta}(g)$, because the eigenvalues of $g^{-1}$ are  the reciprocals of  the eigenvalues of $g$, and $\theta(\phi^i(\lambda^{-1}))=\bar\theta(\phi^i(\lambda))$.  Therefore, $\pi_\theta^\dagger$ is isomorphic to $\pi_{\bar \theta}$.\\

Since $P_{n-1,n}=F^\times \times \Mir_n$, we get $P_{1,n}=P_{n-1,n}^\dagger=F^\times \times \Mir_n^\dagger$, where $F^\times$ is the group of scalar matrices in $G_n$, and
\[ \Mir_n^\dagger = \{\bbsm 1 & v\\ & h \besm  \colon v \in M_{1 \times n-1}(F), h \in G_{n-1} \}. \]
Similarly, 
\[P_{1,n}(1)=F^\times \times (\Mir_{n-1} \ltimes F^{n-1} )^\dagger.\]
\begin{lemma} \label{Pi_1_n_Mir} The representations
$\Pi_{1,n}$ of $P_{1,n}=F^\times \times \Mir_n^\dagger$, and 
$\Pi_{1,n}^\dagger$ of $P_{n-1,n}=F^\times \times \Mir_n$ have dimension $(q-1)(q^2-1) \cdots(q^{n-1}-1)$ and are isomorphic to:
\begin{enumerate}
    \item[\emph{(1)}]  $\Pi_{1,n}^\dagger= 1_{F^\times} \boxtimes (\Phi^+)^{n-1}(1_{\Mir_1})$.
    \item[\emph{(2)}] $\Pi_{1,n}=1_{F^\times} \boxtimes \Ind_{U_n}^{\Mir_n^\dagger}(\vartheta)$, where $U_n$ is the  unipotent subgroup of $\Mir_n^\dagger$ consisting of upper triangular matrices with ones on the diagonal, and $\vartheta(g)=\psi_0(g_{1,2}+ \dots+g_{n-1,n})$ with $\psi_0$ a non-trivial additive character of $F$.
    \end{enumerate}
\end{lemma}
\begin{proof}
    (1) We prove this by induction on $n$ starting with the base case $n=2$. By Lemma \ref{Pi_n-1_n} we have
    \[ \Pi_{1,2} = 1_{F^\times} \boxtimes \Ind_{\Mir_1 \ltimes F}^{\Mir_2} (\psi_1 \otimes 1_{\Mir_1}) = \Ind_{F^\times \times (\Mir_1 \ltimes F)}^{F^\times \times  \Mir_2} (1_{F^\times} \boxtimes \psi_1).\] We note that the subgroup of $G_2$ given by  \[ F^\times \times (\Mir_1 \ltimes F) = \{\bbsm a & b\\&a \besm  \colon a \in F^\times, b \in F\},\]  is preserved by the $\dagger$ involution. Also, $F^\times \times \Mir_2 = P_{1,2}$ and hence $(F^\times \times \Mir_2)^\dagger = P_{1,2}^\dagger = P_{1,2}$. Therefore,
    \[ 
        \Pi_{1,2}^\dagger = \Ind_{F^\times \times (\Mir_1 \ltimes F)}^{P_{1,2}} (1_{F^\times} \boxtimes \bar{\psi_1}) =1_{F^\times} \boxtimes \Ind_{\Mir_1 \ltimes F}^{\Mir_2} (\bar{\psi_1}) \simeq 1_{F^\times} \boxtimes \Phi^+(1_{\Mir_1}).
    \]
    We assume inductively that $\Pi_{1,n-1}^\dagger= 1_{F^\times} \boxtimes (\Phi^+)^{n-2}(1_{\Mir_1})$. By definition, we have   $ \Pi_{1,n} = \Ind_{P_{1,n}(1)}^{P_{1,n}} (\psi_1 \otimes \Pi_{1,n-1})$. Using  
     $P_{1,n}=F^\times \times \Mir_n^\dagger$ and $P_{1,n}(1)^\dagger=F^\times \times (\Mir_{n-1} \ltimes F^{n-1} )$, we get:
     \begin{align*} \Pi_{1,n}^\dagger &=  \Ind_{F^\times \times (\Mir_{n-1} \ltimes F^{n-1} )}^{F^\times \times \Mir_n}(1_{F^\times} \boxtimes (\bar{\psi_1} \otimes (\Phi^+)^{n-2}(1_{\Mir_1}) ))\\&=1_{F^\times} \boxtimes \Ind_{\Mir_{n-1} \ltimes F^{n-1} }^{\Mir_n} (\bar{\psi_1} \otimes (\Phi^+)^{n-2}(1_{\Mir_1})) \\
     &\simeq  1_{F^\times} \boxtimes (\Phi^+)^{n-1}(1_{\Mir_1}).
     \end{align*}
    
    (2) As mentioned at the beginning of the current subsection,  the representation $(\Phi^+)^{n-1}(1_{\Mir_1})$ of $\Mir_n$ is isomorphic to  $\Ind_{U_n}^{\Mir_n}(\vartheta)$ where $\vartheta(g) = \tilde\psi_0(g_{1,2}+ \dots+g_{n-1,n})$, and  $\tilde\psi_0$ is any non-trivial additive character  of $F$. We take $\tilde\psi_0=\bar{\psi_0}$. Therefore, we have 
    \[\Pi_{1,n}^\dagger= 1_{F^\times} \boxtimes \Ind_{U_n}^{\Mir_n}(\vartheta).\]
    Since $U_n^\dagger=U_n$, we get
\[\Pi_{1,n}= 1_{F^\times} \boxtimes \Ind_{U_n}^{\Mir_n^\dagger }(\bar{\vartheta}),\]
    where $(\bar \vartheta)(g) = \psi_0(g_{1,2}+ \dots+g_{n-1,n})$.
    \end{proof}

The next result gives the dimension of $\Pi_{k,n}$:
\begin{lemma} \label{dim_Pi_k_n}
    $\dim \Pi_{k,n}= (q^k-1)(q^{k+1}-1) \cdots (q^{n-1}-1)=\prod_{i=k}^{n-1}(q^i-1)$.
\end{lemma}
\begin{proof}
   If $k=n$ we take the above product to be trivial, so that dim$(\Pi_{k,n})=1$. We prove the formula by induction on $n$. For the base case $n=1$, we must have $k=1$ and hence  dim$(\Pi_{1,1})=1$. For $n>1$, the dimension of the right-side of the recursive formula for $(\Pi_{k,n})$ is  
   \[\sum_r \tfrac{|P_{k,n}|}{|P_{k,n}(r)|} \dim(\St_{G_r}) \dim(\Pi_{r,n-k}).\]
   It is easily seen that 
   \[        \tfrac{|P_{k,n}|}{|P_{k,n}(r)|}=\tfrac{|G_k \times G_{n-k}|}{|M_r|}=\tfrac{|G_k \times G_{n-k}|}{|P^*_{r,k} \times P_{r,n-k}|} \tfrac{|P^*_{r,k} \times P_{r,n-k}|}{|M_r|}
=\stirling{k}{r}_q \stirling{n-k}{r}_q |G_r|=a(k \times n-k,r,q)\]
where, in the first equality we have used the fact that 
   $P_{k,n}(r)=M_{r} \ltimes N_{k,n}$ and $P_{k,n}=(G_k \times G_{n-k})  \ltimes N_{k,n}$, and in the third equality we have used the fact that $|G_k|/|P^*_{r,k}|$ equals the number $\stirling{k}{k-r}_q=\stirling{k}{r}_q$ of $(k-r)$-dimensional subspaces of $F^k$, and $|G_{n-k}|/|P_{r,n-k}|$ equals the number $\stirling{n-k}{r}_q$ of $r$-dimensional subspaces of $F^{n-k}$.
Together with the inductive hypothesis dim$(\Pi_{r,n-k})=\prod_{i=r}^{n-k-1}(q^i-1)$, and the fact that dim$(\St_{G_r})=q^{\binom{r}{2}}$, the recursive formula for dim$(\Pi_{k,n})$    
 simplifies to 
\[\sum_r  a(k \times n-k,r,q)  q^{\tbinom{r}{2}} \prod_{i=r}^{n-k-1}(q^i-1)= q^{\binom{n-k}{2}} \sum_r  a(k \times n-k,r,q) (q^{-n+k+1};q)_{n-k-r}.  \]
By Lemma \ref{amnr_lemma}, this  equals
\[ q^{\binom{n-k}{2}} q^{k(n-k)} (q^{-n+1};q)_{n-k}=(q^k-1)(q^{k+1}-1)\cdots (q^{n-1}-1).\]
\end{proof}

\begin{corollary} \label{Pi_k_n_irreducibility}
    The representation $\Pi_{k,n}$ of $P_{k,n}$ is irreducible if and only if $k \in \{1,n-1,n\}$.
\end{corollary}
\begin{proof}
    If $k=n$, then $\Pi_{k,n}$ is irreducible as it is the trivial representation. We now assume $1 \leq k \leq n-1$. Let $\mu= \min\{k,n-k\}$. The twisted Jacquet modules $(\tau_1, \dots, \tau_\mu)$ of  $\tau=\Pi_{k,n}$ are given by  $\tau_r \simeq \St_{G_r} \otimes \Pi_{r,n-k}$. By Lemma \ref{dim_Pi_k_n}, these modules are non-zero as their dimension is positive. Therefore, by Corollary \ref{Cat_equiv_corollary}, $\Pi_{k,n}$ is irreducible if and only if  $\mu=1$, i.e., $k \in \{1, n-1\}$, and $\tau_1$ is irreducible. If $k=n-1$, we have $\tau_1 \simeq \Pi_{1,1}$ which is trivial and hence irreducible. As for $k=1$, we prove that $\Pi_{1,n}$ is irreducible by induction on $n$. If $n=1$, then $\Pi_{1,1}$ is trivial and hence irreducible. We assume inductively that $\Pi_{1,n-1}$ is irreducible. The twisted Jacquet module $\tau_1$ for $\tau=\Pi_{1,n}$ is isomorphic to $\Pi_{1,n-1}$, which is irreducible by the inductive hypothesis. This completes the proof that $\Pi_{1,n}$ is irreducible. 
\end{proof}

In the  next result,  we determine the character of $\Pi_{k,n}$ 
\underline{only for the special type} of matrices $g = \bbsm g_1 & g_1 g_3 \\ 0 & g_2 \besm$ for which the semisimple part of $g_1 \in G_k$ comes from $F_k^\times$. We denote the character of $\Pi_{k,n}$ as $\Theta_{k,n}$.
\begin{theorem} \label{Pi_k_n_char}
    Let  $g = \bbsm g_1 & g_1 g_3\\0& g_2 \besm \in P_{k,n}$. If the semisimple part of $g_1 \in G_k$   comes from $F_k^\times$, then the character $\Theta_{k,n}$ of $\Pi_{k,n}$ at $g$ is zero, unless the semisimple part of $g$ comes from $F_d^\times$ where $d = \text{gcd}\{k,n\}$, in which case
    \[\Theta_{k,n}(g)=(-1)^{n-k} ({q'}^{t(g_1)};q')_{t(g)-t(g_1)}, \]
    where $F[\lambda]=F_m \subset F_d$ for any eigenvalue $\lambda$ of $g$, $q'=q^m$ and $t(g)=\dim_{F_n}\ker(g-\lambda I_n)$ and $t(g_1)=\dim_{F_k}\ker(g_1-\lambda I_k)$.
\end{theorem} 
The proof of this result requires several arguments, and the next subsection \ref{sec_proof_Pi_k_n_char}  is devoted to the proof.
\subsection{Proof of Theorem \ref{Pi_k_n_char}}  \label{sec_proof_Pi_k_n_char} \hfill \\
Throughout this proof let $s_1, s_2$ and $s$ denote the semisimple parts of $g_1, g_2$ and $g$ respectively. Since $s_1$ comes from $F_k^\times$, its minimal polynomial which we denote $p(X)$ is irreducible in $F[X]$ of degree $m|k$. Let $\lambda \in F_m$ denote an eigenvalue of $g_1$, so that $p(X)$ factorizes as $p(X)=\prod_{i=0}^{m-1}(X-\phi^i(\lambda))$ in $F_m[X]$. The symbols $r',k', n'$ and $q'$ stand for $r/m,k/m, n/m$ and $q^m$ respectively. \\

We prove the assertion by induction on the `length' $(n-k)$. If $n-k=0$, the semisimple part of $g=g_1$ indeed comes from $F_d^\times$,  and also  $t(g)-t(g_1)=0$. Therefore $(-1)^{n-k} ({q'}^{k'};q')_{t(g)-t(g_1)}=(-1)^0(q^m;q^m)_0 = 1$, which agrees with the fact that $\Pi_{k,k}$ is the trivial representation of $P_{k,k}=G_k$.
For $n-k>0$, we assume inductively that the assertion is true for $\Pi_{r,n-k}$ which has `length' $(n-k-r) <n-k$. First we prove the theorem for the case 
when $s$ does not come from $F_n^\times$.  \\

\noindent\textbf{\underline{Step 1}: Proof of Theorem \ref{Pi_k_n_char} in  case  $s$ does not come from $F_n^\times$} \hfill \\


We start with the following  elementary lemma:
\begin{lemma} \label{jordan_chevalley_for_P_k_n}
    Let $g=\bbsm g_1 & g_1 g_3\\0& g_2 \besm \in P_{k,n}$ and let $g=s \cdot u$ be the Jordan decomposition of $g$ into its semisimple and unipotent parts. If $s = \bbsm s_1 & s_3\\&s_2 \besm$ and $u = \bbsm u_1 & u_3\\&u_2 \besm$, then
    \begin{enumerate}
        \item[\emph{(1)}]
    For $i \in\{1,2\}$, we have $s_i$  are semisimple, $u_i$ are unipotent and $g_i=s_i u_i=u_i s_i$ is   the Jordan decomposition of $g_i$ into semisimple and unipotent parts.\\

    \item[\emph{(2)}] $s$ comes from $F_n^\times$ if and only if $s_1$ and $s_2$ come from $F_k^\times$ and  $F_{n-k}^\times$ respectively, and have the same minimal polynomial of degree $m$ dividing $\text{gcd}\{k,n\}$.
    \end{enumerate}
\end{lemma}
\begin{proof}
   (1)  Let $m$ be a positive integer such that $u^m = I_n$. Since  $u^m= \bbsm u_1^m & \ast\\ & u_2^m \besm$, it follows that $u_1$ and $u_2$ are unipotent. Since submodules and quotient modules of semisimple modules are themselves semisimple, it follows that  $s_1$  and $s_2$ are semisimple.
    Also the condition $s u=u s = g$ gives $s_i u_i=u_i s_i = g_i$ for $i=1,2$.   This shows that $g_i = s_iu_i$ is the Jordan decomposition of $g_i$ into their semisimple and unipotent parts.\\

    (2) The characteristic polynomial of $s$ which is $\det\bbsm (X I_k-s_1) & -s_3 \\ & (X I_{n-k} -s_2) \besm$ equals $\det(X I_k-s_1) \cdot \det(X I_{n-k}-s_2)$. By definition $s$ comes  from $F_n^\times$ if and only if  its characteristic polynomial $\det(X I_n-s)= 
      \det(X I_k-s_1) \cdot \det(X I_{n-k}-s_2)$ is a   power of an irreducible polynomial $P(X)$ of degree $m|n$. This is possible if and only if $\det(X I_{n-k}-s_2) = P(X)^{(n-k)/m}$ and $\det(X I_{k}-s_1) = P(X)^{k/m}$, or equivalently $s_1$ and $s_2$ come from $F_k^\times$ and $F_{n-k}^\times$ and have the same minimal polynomial. Clearly $m$ divides $k, n-k$ and hence $\text{gcd}\{k,n\}$.
     \end{proof}
    
By definition \ref{def_Pi} of $\Pi_{k,n}$, its character  at $g = \bbsm g_1 & g_1g_3\\ & g_2 \besm\in P_{k,n}$ is 
\begin{align} \label{eq:formula_char_Pi_k_n}
    \nonumber \Theta_{k,n}(g) &= \sum_{r \geq 1}  \chi_{\Ind_{P_{k,n}(r)}^{P_{k,n}} ( \psi_r \otimes [\St_{G_r} \otimes \Pi_{r,n-k}])
    }(g) \\ &= \sum_{r \geq 1}
\tfrac{1}{|P_{k,n}(r)|} \sum_{t \in P_{k,n}}
\chi_{\psi_r \otimes [\St_{G_r} \otimes \Pi_{r,n-k}]} (t^{-1} g t),
\end{align} 
where the inner sum $\sum_{t \in P_{k,n}}$ runs over those $t$ such that $t^{-1} g t \in P_{k,n}(r)$. 

\begin{proposition} \label{char_Pi_k_n_not_semisimple}
For $g = \bbsm g_1 & g_1 g_3 \\ 0 & g_2 \besm$ for which the semisimple part of $g_1 \in G_k$ comes from $F_k^\times$, the character $\Theta_{k,n}$ of $\Pi_{k,n}$ at $g$ is $0$ if the semisimple part of $g$ does not come from $F_n^\times$.
\end{proposition} 
\begin{proof}
 We prove the assertion  by induction on the `length' $(n-k)$. If $n-k=0$, the assertion is vacuously true, because $g=g_1$. For $n-k>0$, we assume inductively that the assertion is true for $\Pi_{r,n-k}$ which has `length' $(n-k-r) <n-k$.  Let $s_1, s_2$ and $s$ denote the semisimple parts of $g_1, g_2$ and $g$. Therefore, by Lemma \ref{jordan_chevalley_for_P_k_n},  $s$ does not come from $F_n^\times$ if and only if: either $s_2$ does not come from $F_{n-k}^\times$ or  $s_2$ comes  from $F_{n-k}^\times$ but with a minimal polynomial different from $p(X)$. In 
\eqref{eq:formula_char_Pi_k_n}, if $t=\bbsm t_1 & t_1t_3\\ & t_2 \besm$ then $t^{-1} g t$ is of the form $\bbsm t_1^{-1} g_1 t_1 & \ast \\ & t_2^{-1} g_2 t_2 \besm$. It is in $P_{k,n}(r)$ and hence is of the form $\bbsm \begin{smallmatrix}
h_1 &\\h_5 &h_4 \end{smallmatrix} & *\\ & \begin{smallmatrix}
h_1 &h_3\\ &h_2 \end{smallmatrix}\besm$.
If there is no $t \in P_{k,n}$ such that $t^{-1} g t$ is in $P_{k,n}(r)$, then clearly 
$\Theta_{k,n}(g)=0$ by \eqref{eq:formula_char_Pi_k_n}. So we now assume there is such a $t$. We are given that  the semisimple part $s_1$ of $g_1$ comes from $F_k^\times$ and has minimal polynomial $p(X)$. Thus the semisimple part of  $t_1^{-1} g_1 t_1 = \bbsm 
h_1 & \\ h_5 & h_4 \besm$
also comes from $F_k^\times$ and has minimal polynomial $p(X)$. 
As in Lemma \ref{jordan_chevalley_for_P_k_n},
this is equivalent to the semisimple part of $h_1, h_4$ coming from $F_{\text{gcd}\{r,k\}}^\times$. In particular, the semisimple part $\sigma_1$ of $h_1$ comes from $F_r^\times$ and has minimal polynomial $p(X)$. Turning now to 
$t_2^{-1} g_2 t_2= \bbsm h_1 & h_3\\ & h_2 \besm$, and using the fact that   the semisimple part $\sigma_1$ of $h_1$ comes from $F_r^\times$ with minimal polynomial $p(X)$, it follows that in case $s_2$ comes from $F_{n-k}^\times$, its minimal polynomial must be the same as that of $s_1$, namely $p(X)$. Therefore, by  Lemma \ref{jordan_chevalley_for_P_k_n}, that  the condition $g$ does not come from $F_n^\times$ is equivalent to the semisimple part $s_2$ of $g_2$ not coming from $F_{n-k}^\times$. So the semisimple part of $\bbsm h_1 & h_3\\ & h_2 \besm$ does not come from $F_{n-k}^\times$ but the semisimple part of $h_1$ does come from $F_r^\times$. Since we have assumed the inductive hypothesis for $\Pi_{r,n-k}$, we conclude that $\Theta_{r,n-k}(t_2^{-1} g_2 t_2)=0$. This completes the proof that $\Theta_{k,n}(g)=0$ if $s$ does not come from $F_n^\times$.
\end{proof}


\noindent \textbf{\underline{Step 2}: Representatives $g_{\alpha,\beta, h}$ of the left cosets of $P_{k,n}(r)$ in $P_{k,n}$} \hfill \\

As shown in Lemma \ref{dim_Pi_k_n}, the index of $P_{k,n}(r)$ in $P_{k,n}$ is
\[\tfrac{|P_{k,n}|}{|P_{k,n}(r)|}=\tfrac{|G_k \times G_{n-k}|}{|P^*_{r,k} \times P_{r,n-k}|} \tfrac{|P^*_{r,k} \times P_{r,n-k}|}{|M_r|}
=\stirling{k}{r}_q \stirling{n-k}{r}_q |G_r|=a(k \times n-k,r,q).\]
Since 
\[M_r=\left\{ \bbsm \begin{smallmatrix}
h_1 &\\h_5 &h_4 \end{smallmatrix} & \\ & \begin{smallmatrix}
h_6 &h_3\\ &h_2 \end{smallmatrix}\besm  \in P^*_{r,k} \times P_{r,n-k} \colon h_1=h_6\right\}, \]
it is clear that the set
\[ \left\{ \bbsm \begin{smallmatrix}
h &\\ &I_{k-r} \end{smallmatrix} & \ \\ & I_{n-k} \besm  \in P^*_{r,k} \times P_{r,n-k} \colon h \in G_r\right\}, \]
is a complete set of representatives of the left cosets of $M_r$ in $P^*_{r,k} \times P_{r,n-k}$.\\

For each $(k-r)$ dimensional subspace $\alpha$ of $F^k$, let $g_\alpha \in G_k$ be a matrix whose last $(k-r)$ columns span $\alpha$. For each $r$-dimensional subspace $\beta$ of $F^{n-k}$, let $g_\beta \in G_{n-k}$ be a matrix whose first $r$ columns span $\beta$. We claim that the collection of $\stirling{k}{k-r}_q \stirling{n-k}{r}_q$ matrices in $G_k \times G_{n-k} \subset P_{k,n}$ of the form  $g_{\alpha,\beta}=\bbsm  g_\alpha   & \\& g_\beta \besm$ 
forms a complete set of representatives of the  $\stirling{k}{k-r}_q \stirling{n-k}{r}_q$ left cosets of $P^*_{r,k} \times P_{r,n-k}$ in $G_k\times G_{n-k}$. It  suffices to show that the cosets $g_{\alpha, \beta} P^*_{r,k} \times P_{r,n-k}$  are distinct. However, this is immediate from the fact that the coset $g_{\alpha, \beta} P^*_{r,k} \times P_{r,n-k}$ consists of those elements $\bbsm g_1 &\\ & g_2 \besm $ of $G_k \times G_{n-k}$ for which the span of the last $k-r$ columns of $g_1$ is $\alpha$, and the span of the first $r$ columns of $g_2$ is $\beta$. In conclusion, the matrices $g_{\alpha,\beta, h}=\bbsm g_\alpha \bbsm h &\\&I_{k-r} \besm & \\& g_\beta \besm $ form a complete set of representatives for the left cosets of $P_{k,n}(r)$ in $P_{k,n}$. 
The matrix 
\[ g_{\alpha,\beta, h}^{-1} \,g \,g_{\alpha,\beta, h}= \bbm  \bbsm h^{-1} & \\ & I_{k-r} \besm g_\alpha^{-1} g_1 g_\alpha \bbsm h &\\&I_{k-r} \besm &   \\ & g_\beta^{-1} g_2 g_\beta \bem  \cdot \bbm I_k & \bbsm h^{-1} & \\ & I_{k-r} \besm g_\alpha^{-1} g_3 g_\beta \\ & I_{n-k} \bem 
\]
is in $P_{k,n}(r)$ if and only if 
\[  \bbm  \bbsm h^{-1} & \\ & I_{k-r} \besm g_\alpha^{-1} g_1 g_\alpha \bbsm h &\\&I_{k-r} \besm &   \\ & g_\beta^{-1} g_2 g_\beta \bem \text{ has the form }
\bbsm \begin{smallmatrix}
h_1 &\\h_5 &h_4 \end{smallmatrix} & \\ & \begin{smallmatrix}
h_1 &h_3\\ &h_2 \end{smallmatrix}\besm  \in M_r.\]

We note that  $h_4 \in G_{k-r}$ represents $g_1|_\alpha$  and $h_1$ represents the induced operator  $\bar{g_1}: F^k/\alpha \to F^k/\alpha$ as well as the operator  $g_2|_\beta: \beta \to \beta$. We now rewrite the inner sum (for fixed value of $r$ in the outer sum) in the formula \eqref{eq:formula_char_Pi_k_n} in terms of the coset representatives $g_{\alpha, \beta, h}$:
\begin{equation} \label{eq:ind_char_for_Pi} \sum_{g_{\alpha, \beta,h}} \chi_{\psi_r \otimes [\St_{G_r} \otimes \Pi_{r,n-k}]} (g_{\alpha,\beta, h}^{-1} \,g \,g_{\alpha,\beta, h}),\end{equation}
where $\sum_{g_{\alpha, \beta, h}}$ runs over those $g_{\alpha,\beta, h}$ for which $g_{\alpha,\beta, h}^{-1} \,g \,g_{\alpha,\beta, h}$ lies in $P_{k,n}(r)$. \\


\noindent\textbf{\underline{Step 3}: The number of choices of $g_{\alpha, \beta,h}$  that contribute to \eqref{eq:ind_char_for_Pi}}\hfill \\

Since $\St_{G_r}(h_1)=0$ for non-semisimple $h_1$, we  insist the operators $\bar{g_1}: F^k/\alpha \to F^k/\alpha$ and  $g_2|_\beta: \beta \to \beta$ are semisimple, for otherwise $\St_{G_r}(h_1)=0$ (in \eqref{eq:ind_char_for_Pi}) for the matrix $h_1$ that represents these operators.  Since the case when $s$ does not come from $F_n^\times$ has been addressed in Proposition \ref{char_Pi_k_n_not_semisimple}, we now assume $s$ comes from $F_n^\times$. By part (2) of Lemma \ref{jordan_chevalley_for_P_k_n}, we must have that $s_2$ comes from $F_{n-k}^\times$ and its minimal polynomial is again $p(X)$. Therefore the semisimple part of  $g_\beta^{-1} g_2 g_\beta=\bbsm h_1 &h_3\\ &h_2  \besm$ also comes from $F_{n-k}^\times$. Again by part (2) of Lemma \ref{jordan_chevalley_for_P_k_n}, the semisimple part of $h_1$ comes from $F_r^\times$ and has minimal polynomial $p(X)$. Since $h_1$  is semisimple, we conclude that $h_1$ is \emph{semisimple  and comes from $F_r^\times$}. 
 Since $h_1$ represents the operators  $\bar{g_1}: F^k/\alpha \to F^k/\alpha$ and  $g_2|_\beta: \beta \to \beta$, both these operators are semisimple and  come from $F_r^\times$. \\

We consider $F^k, F^{n-k}$ and $F^n$ as  $F[X]$-modules with $X$ acting as the $F$-linear endomorphism $g_1, g_2$ and $g$ respectively. Since the minimal polynomial of $s_1, s_2$ and $s$ is $p(X)$, the characteristic polynomials of $s_1, s_2$ and $s$ (and hence the characteristic polynomials of  $g_1, g_2$  and $g$) are  $p(X)^{k'}$, $p(X)^{n'-k'}$ and $p(X)^{n'}$ respectively. In particular  we see $F^k, F^{n-k}$ and $F^n$ are  $p(X)$-torsion modules. Thus we can write 
\[
F^k \simeq_{F[X]} \bigoplus_{i=1}^{t(g_1)} \, \sfrac{F[X]}{p(X)^{\mu_i}},\;
F^{n} \simeq_{F[X]} \bigoplus_{i=1}^{t(g)} \,  \sfrac{F[X]}{p(X)^{\epsilon_i}}, \;
F^{n-k} \simeq_{F[X]} \bigoplus_{i=1}^{t(g_2)} \, \sfrac{F[X]}{p(X)^{\nu_i}}
\]
for some positive integers $\nu_i$, $\mu_i$ and $\epsilon_i$. We note that $(X-\lambda)$-torsion submodule of $F^{k}\otimes F_m$ is $\oplus_{i=1}^{t(g_1)} F_m[X]/(X-\lambda)^{\mu_i}$  and hence $t(g_1)=\dim \ker(g_1-\lambda I_{k})$ where the dimension is over any extension field of $F_m$, for example $F_n$. Similarly
\[
  t(g_1)=  \dim_{F_n} \ker(g_1-\lambda I_{k}),\;
  t(g)=  \dim_{F_n} \ker(g-\lambda I_{n}),\;
  t(g_2)=\dim_{F_n} \ker(g_2-\lambda I_{n-k}).
  \]
Consider the multiplication by $p(X)$-maps on each of the three $F[X]$-modules $F^k, F^n$ and $F^{n-k}$. These fit into the following commutative diagram of $F[X]$-modules where the rows are exact:

\begin{equation} \label{eq:snake}
\begin{tikzcd}
0 \arrow[r] & K(g_1) \arrow[r] \arrow[d] & K(g) \arrow[r] \arrow[d] & K(g_2) \arrow[d]    \\
0 \arrow[r] & F^k \arrow[r] \arrow[d,"p(X)"] & F^n \arrow[r] \arrow[d,"p(X)"] & F^{n-k} \arrow[r] \arrow[d,"p(X)"] & 0 \\
0 \arrow[r] & F^k \arrow[r]                  & F^n \arrow[r]                  & F^{n-k} \arrow[r]                  & 0
\end{tikzcd}
\end{equation}
and the kernel of the multiplication by $p(X)$-maps are:

\begin{align*} 
K(g_1) &=\oplus_{i=1}^{t(g_1)} \,  \sfrac{p(X)^{\mu_i-1} F[X]}{p(X)^{\mu_i}}  \,  \simeq \, \oplus^{t(g_1)} \, \sfrac{F[X]}{p(X)} \\
K(g) &= \oplus_{i=1}^{t(g)}\,  \sfrac{p(X)^{\epsilon_i-1} F[X]}{p(X)^{\epsilon_i}}  \,   \simeq\,   \oplus^{t(g)} \, \sfrac{F[X]}{p(X)},\\
K(g_2) &= \oplus_{i=1}^{t(g_2)} \,   \sfrac{p(X)^{\nu_i-1}F[X]}{p(X)^{\nu_i}} \, \simeq \oplus^{t(g_2)} \sfrac{F[X]}{p(X)}
\end{align*}
We also note the structure of  coker$(F^k \xrightarrow{p(X)} F^k)$: since $p(X) F^k = \oplus_{i=1}^{t(g_1)} \,  \sfrac{p(X)  F[X]}{p(X)^{\mu_i}}$, we see that 
\[ \text{coker}(F^k \xrightarrow{p(X)} F^k) \simeq_{F[X]} \,   \oplus^{t(g_1)} \, \sfrac{F[X]}{p(X)}.     \]
For the  $F[X]$-submodule $\alpha$ of $F^k$, we note that $\det(XI_k - g_1)=p(X)^{k'}$ is the product of the characteristic polynomials of $g_1|_\alpha$ and $\bar{g_1}:F^k/\alpha \to F^k/\alpha$. In particular,  $F^k/\alpha$ is $p(X)$-torsion. Since $F^k/\alpha$ is semisimple, it follows that $F^k/\alpha \simeq_{F[X]} \oplus^{r'} F[X]/p(X)$,  and hence $p(X) (F^k/\alpha)=0$. This shows that $p(X) F^k \subset \alpha$, and hence $\bar{\alpha}:=\sfrac{\alpha}{p(X)F^k}$ is a submodule of $\text{coker}(F^k \xrightarrow{p(X)} F^k)$. Therefore, 
\[ \oplus^{r'} \sfrac{F[X]}{p(X)} \simeq \sfrac{F^k}{\alpha} \simeq \text{coker}(F^k \xrightarrow{p(X)} F^k)/\bar\alpha  \simeq (\oplus^{t(g_1)} \, \sfrac{F[X]}{p(X)})/\bar{\alpha}. \]

We note that the $F[X]$-module $\oplus^{\ell} F[X]/p(X)$ is annihilated by $p(X)$, and hence it naturally carries an  $F[X]/p(X)$-module structure. Identifying $F[X]/p(X)$ with the field $F_m$, we identify the $F[X]$-module $\oplus^{\ell} F[X]/p(X)$ with the $F_m$-vector space $F_m^\ell$. In this way 
\[F^k/\alpha \simeq F_m^{r'} \simeq F_m^{t(g_1)}/\bar\alpha. \]
This shows that the number of ways to choose $\alpha$ is same as the number $\stirling{t(g_1)}{t(g_1)-r'}_{q'}$ of $(t(g_1)-r')$-dimensional $F_m$-linear subspaces  of $F_m^{t(g_1)}$.\\

Similarly, the characteristic polynomial of $g_2|_\beta$ is $p(X)^{r'}$, and since $g_2|_\beta$ is semisimple, we see that $\beta \simeq \oplus^{r'} F[X]/p(X)$. Clearly $p(X) \beta=0$ and hence $\beta$ is a submodule of $K(g_2) \simeq \oplus^{t(g_2)} F[X]/p(X)$.
As above, we may regard $\beta$ as $r'$ dimensional $F_m$-linear subspace of $F_m^{t(g_2)}$, and hence the number of ways of choosing $\beta$ is $\stirling{t(g_2)}{r'}_{q'}$.\\

Let $h_1$ be the $r \times r$ submatrix  on the first $r$ rows and columns of $g_\beta^{-1} g_2 g_\beta$ as well as $\bbsm h^{-1} & \\ & I_{k-r} \besm g_\alpha^{-1} g_1 g_\alpha \bbsm h &\\&I_{k-r} \besm $ which has the form $\bbsm h_1 & \\\ast & \ast \besm$. If $h_1'$  is the $r \times r$ submatrix  on the first $r$ rows and columns of $ g_\alpha^{-1} g_1 g_\alpha$, then we need $h^{-1}h_1'h=h_1$. Since $h_1'$ represents $\bar{g_1}:F^k/\alpha \to F^k/\alpha$ and $h_1$ represents $g_2|_\beta: \beta \to \beta$, it follows that $h$ is an invertible element of 
$\Hom_{F[X]}(\beta, F^k/\alpha)$. Since $\beta \simeq_{F[X]} F^k/\alpha \simeq_{F[X]} \oplus^{r'} F[X]/p(X)$, we can identify $\Hom_{F[X]}(\beta, F^k/\alpha)$  with $\End_{F_m}(F_m^{r'})$ and invertible elements of $\Hom_{F[X]}(\beta, F^k/\alpha)$ with $\GL_{r'}(F_m)$. Thus, the number of choices for $h$ is $\# \GL_{r'}(F_m)$. \\

 We now summarize the above discussion about the number of ways to pick $\alpha, \beta$ and $h$: Of  the $a(k \times n-k,r,q)$ cosets $g_{\alpha, \beta, h}$ of $P_{k,n}(r)$ in $P_{k,n}$, for $g$ as above, there are  only 
\[ a(t(g_1) \times t(g_2), r',q')=\stirling{t(g_1)}{r'}_{q'} \cdot \stirling{t(g_2)}{r'}_{q'}\cdot |\GL_{r'}(F_m)|\]
cosets that contribute to \eqref{eq:ind_char_for_Pi}.\\

For use in the next step, we count the number of choices $\stirling{t(g_1)}{r'}_{q'} \cdot \stirling{t(g_2)}{r'}_{q'}$ of the pair $(\alpha, \beta)$ in a finer way:   Let coker$(p(X))$ stand for $\text{coker}(F^k\xrightarrow{p(X)}F^k)$. From the `snake-lemma' applied to the diagram \eqref{eq:snake}, there is a $F[X]$-module connecting homomorphism $K(g_2) \xrightarrow{C(g)} \text{coker}(p(X))$, such that 
\[ K(g)/K(g_1) \simeq_{F[X]}   \ker(C(g)).\]
Since $\beta$ is a submodule of $K(g_2)$ and $F^k/\alpha \simeq \text{coker}(p(X))/\bar{\alpha}$ as above, we define a map $C(g)_{\alpha,\beta}$ by the diagram below:
\[ 
\begin{tikzcd}
K(g_2) 
  \arrow[r, "C(g)"] 
& \text{coker}(p(X))
  \arrow[d, twoheadrightarrow]
\\
\beta \arrow[r, "C(g)_{\alpha, \beta}"'] \arrow[u,hook']
&
F^k/\alpha
\end{tikzcd}
\]
Since $K(g_2) \simeq_{F[X]} \,   \oplus^{t(g_2)} \, \sfrac{F[X]}{p(X)}$ 
 and $\text{coker}(p(X)) \simeq_{F[X]} \,   \oplus^{t(g_1)} \, \sfrac{F[X]}{p(X)}$, the map $C(g)$ is an element of $\Hom_{F[X]}(\oplus^{t(g_2)} \, \sfrac{F[X]}{p(X)}, \oplus^{t(g_1)} \, \sfrac{F[X]}{p(X)})$ which we identify with $\Hom_{F_m}(F_m^{t(g_2)}, F_m^{t(g_1)})$. Let $t(g_3):=\dim_{F_m}\ker(C(g))$. From the fact that $K(g)/K(g_1) \simeq_{F[X]} \ker(C(g))$, we get
 \[ t(g)-t(g_1)=t(g_3).\]
We partition the number of choices $\stirling{t(g_1)}{r'}_{q'} \cdot \stirling{t(g_2)}{r'}_{q'}$ of the pairs $(\alpha, \beta)$ as 
\[ \stirling{t(g_1)}{r'}_{q'} \cdot \stirling{t(g_2)}{r'}_{q'}
 = \sum_s n_r(s),
\]
where $n_r(s)$ is the number of pairs $(\alpha, \beta)$ such  that the map 
\[ C(g)_{\alpha, \beta} \in \Hom_{F[X]}(\beta, F^k/\alpha) \simeq \End_{F_m}(F_m^{r'}), \]
has rank $s$.  We now calculate $n_r(s)$ in terms of $t(g_3)$. Let $\dim_{F_m}( \beta \cap \ker C(g))=r'-t$ for some integer $t$. The number of such $\beta \subset K(g_2)$ is 
\[ \stirling{t(g_3)}{r'-t}_{q'} \stirling{t(g_2)-t(g_3)}{t}_{q'} (q')^{t(t(g_3)-r'+t)}.\] For such a $\beta$, the map $C(g)_{\alpha, \beta}$ has rank $s$ if and only if $\dim_{F_m}(\bar\alpha \cap \, C(g)(\beta))=t-s$. The number of such $\bar\alpha \subset \text{coker}(p(X))$  is 
\[ \stirling{t}{t-s}_{q'} \stirling{t(g_1)-t}{t(g_1)-r'-t+s}_{q'}  (q')^{(t(g_1)-r'-t+s)s}.\]
Therefore, 
\begin{equation} \label{eq:n_r_s}
     n_r(s)=\sum_t \left(\stirling{t(g_3)}{r'-t}_{q'} \stirling{t(g_2)-t(g_3)}{t}_{q'} (q')^{t(t(g_3)-r'+t)} \right) \left( \stirling{t}{t-s}_{q'} \stirling{t(g_1)-t}{t(g_1)-r'-t+s}_{q'}  (q')^{(t(g_1)-r'-t+s)s}\right).     
\end{equation}

\noindent{\textbf{\underline{Step 4}: The sum over $h$ in \eqref{eq:ind_char_for_Pi} for fixed $\alpha, \beta$} \hfill \\

In the formula \eqref{eq:ind_char_for_Pi}, we first fix a pair $(\alpha, \beta)$ such that the $F[X]$-submodule $\beta$ of $F^{n-k}$ and the quotient $F[X]$-module  $F^k/\alpha$ are isomorphic to $\oplus^{r'} F[X]/p(X)$, and we  calculate the sum  
\[ \sum_{h} \chi_{\psi_r \otimes [\St_{G_r} \otimes \Pi_{r,n-k}]} (g_{\alpha,\beta, h}^{-1} \,g \,g_{\alpha,\beta, h})\]
where $h$ runs over the $F[X]$-module isomorphisms from $\beta$ to $F^k/\alpha$. The above sum can be written as
\[ \sum_{h} \St_{G_r}(h_1) \cdot \Theta_{r,n-k}(g_\beta^{-1} g_2 g_\beta) \cdot  \psi_0(\tr(h^{-1} g_3(\alpha, \beta)), \]
where $h_1$ as before, is the $r \times r$ submatrix of $g_\beta^{-1} g_2 g_\beta$ on the first $r$ rows and columns, and where $g_3(\alpha, \beta)$ is the $r \times r$ submatrix of $g_\alpha^{-1} g_3 g_\beta$ on the first $r$ rows and columns.
We note that $\St_{G_r}(h_1)= (-1)^{r-r'}  {q'}^{\binom{r'}{2}}$ by Lemma \ref{St_char} and the fact that $h_1$ is semsimple and comes from $F_r^\times$.
Next we determine $\Theta_{r,n-k}(g_\beta^{-1} g_2 g_\beta)$. The semisimple part  of $g_\beta^{-1} g_2 g_\beta$ is $g_\beta^{-1} s_2 g_{\beta}$ which comes from $F_{n-k}^\times$ because $s_2$ comes from $F_{n-k}^\times$. Also $g_\beta^{-1} g_2 g_\beta=\bbsm h_1 &h_3\\ &h_2  \besm$, with the semisimple part of $h_1$ (which is $h_1$ itself) coming from $F_r^\times$.  Therefore by the inductive hypothesis, we have 
\[ \Theta_{r,n-k} (g_\beta^{-1} g_2 g_\beta)= (-1)^{n-k-r} ((q')^{r'};q')_{t(g_2)-r'}.\]
Therefore, the above sum can be written as:
\[  
(-1)^{n-k-r'} (q')^{\binom{r'}{2}} 
((q')^{r'};q')_{t(g_2)-r'}
\sum_{h}   \psi_0(\tr(h^{-1} g_3(\alpha, \beta)). \]
We will show that the term $\sum_{h}   \psi_0(\tr(h^{-1} g_3(\alpha, \beta))$ equals 
$(-1)^s (q')^{\binom{s}{2} +s(r'-s)} |\GL_{r'-s}(F_m)|$ 
where $s$ is the rank of $C(g)_{\alpha, \beta}$ viewed as an element of $\End_{F_m}(F_m^{r'})$ (as in the previous step). Therefore, the above sum becomes
\begin{equation} \label{eq:sum_over_h} (-1)^{n-k-r'} (q')^{\binom{r'}{2}} 
((q')^{r'};q')_{t(g_2)-r'} (-1)^s (q')^{\binom{s}{2} +s(r'-s)} |\GL_{r'-s}(F_m)| n_r(s) \end{equation}
where $n_r(s)$ is determined in \eqref{eq:n_r_s}. We note that $g_3(\alpha, \beta)$ represents the induced map 
\[ \beta \xrightarrow{g_3|_\beta} F^k/\alpha.\]
We recall that 
\[ K(g_2) \simeq_{F[X]} \oplus^{t(g_2)} \tfrac{F[X]}{p(X)}, \quad \text{coker}(p(X))=\tfrac{F^k}{p(X) F^k} \simeq_{F[X]} \,   \oplus^{t(g_1)} \, \tfrac{F[X]}{p(X)}. \]
It will be useful to choose generators $v_1, \dots, v_{t(g_2)}$ and $u_1, \dots,u_{ t(g_1)}$ such that 
\begin{align} \label{eq:generators_Kg_2}  K(g_2)  \simeq_{F[X]} \bigoplus_{j=1}^{t(g_2)} \tfrac{F[X]}{p(X)} v_j, \qquad  \tfrac{F^k}{p(X) F^k} \simeq_{F[X]} \bigoplus_{i=1}^{t(g_1)} \tfrac{F[X]}{p(X)} u_i, \\
 \nonumber \beta  \simeq_{F[X]} \bigoplus_{j=1}^{r'} \tfrac{F[X]}{p(X)} v_j, \quad  \bar\alpha   \simeq_{F[X]} \bigoplus_{i=r'+1}^{t(g_1)} \tfrac{F[X]}{p(X)} u_i, \quad F^k/\alpha \simeq_{F[X]} \bigoplus_{i=1}^{r'} \tfrac{F[X]}{p(X)} u_i.\end{align}
By extending the scalars from $F$ to $F_m$, the $F_m[X]$-modules  $\beta \otimes F_m \subset K(g_2) \otimes_F F_m$,  and $F_m^k/(\alpha \otimes F_m) \subset F_m^k/p(X) F_m^k$ decompose as a direct sum of eigenspaces for the eigenvalues $\phi^\ell(\lambda),   0 \leq \ell \leq m-1$, which we denote as $\beta (\phi^\ell(\lambda)) \subset K(g_2, \phi^\ell(\lambda))$ and $\alpha'(\phi^\ell(\lambda)) \subset U(\phi^\ell(\lambda))$ respectively:
\begin{align*}
    K(g_2) \otimes_F F_m &\simeq_{F_m[X]} \oplus_{\ell=0}^{m-1} K(g_2, \phi^\ell(\lambda)), \quad K(g_2,\lambda) = \oplus_{j=1}^{t(g_2)} \tfrac{F[X]}{X-\lambda} v_j, \\
    \tfrac{F_m^k}{p(X) F_m^k} &\simeq_{F_m[X]} \oplus_{\ell=0}^{m-1} U(\phi^\ell(\lambda)), \quad U(\lambda)= \oplus_{i=1}^{t(g_1)} \, \tfrac{F[X]}{X-\lambda} u_i, \\
    \beta \otimes_F F_m &\simeq_{F_m[X]} \oplus_{\ell=0}^{m-1} \beta(\phi^\ell(\lambda)), \quad \beta(\lambda) = \oplus_{j=1}^{r'} \tfrac{F[X]}{X-\lambda} v_j, \\
       \tfrac{F_m^k}{\alpha \otimes F_m} &\simeq_{F_m[X]} \oplus_{\ell=0}^{m-1} \alpha'(\phi^\ell(\lambda)), \quad \alpha'(\lambda)= \oplus_{i=1}^{r'} \, \tfrac{F[X]}{X-\lambda} u_i
\end{align*}
Here, $K(g_2, \phi^\ell(\lambda))=\phi^\ell(K(g_2, \lambda))$ and similarly for $U(\phi^\ell(\lambda))$, $\beta(\phi^\ell(\lambda))$ and $\alpha'(\phi^\ell(\lambda))$.  

Viewing the 
connecting homomorphism $C(g) \in \Hom_{F[X]}(K(g_2), \text{coker}(p(X)))$ as an element of $\Hom_{F_m[X]}(K(g_2)\otimes F_m, F_m^k/p(X) F_m^k)$, there is an element $C(g,\lambda) \in \Hom_{F_m[X]}(K(g_2, \lambda), U(\lambda))$ such that $C(g)=\oplus_{\ell=0}^{m-1} K(g_2, \phi^\ell(\lambda)) \xrightarrow{C(g,\phi^\ell(\lambda))} \oplus_{\ell=0}^{m-1} U(\phi^\ell(\lambda))$ with $C(g,\phi^\ell(\lambda)) = \phi^\ell \circ C(g,\lambda) \circ \phi^{-\ell}$. Similarly the map $C(g)_{\alpha, \beta} \in \Hom_{F[X]}(\beta, F^k/\alpha)$ when viewed as an element of $\Hom_{F_m[X]}(\beta \otimes F_m, F_m^k/(\alpha \otimes F_m))$ is described by a map $C(g,\lambda)_{\alpha, \beta} \in \Hom_{F_m[X]}(\beta(\lambda), \alpha'(\lambda))$. In terms of the generators $K(g_2,\lambda) = \bigoplus_{j=1}^{t(g_2)} \tfrac{F[X]}{X-\lambda} v_j$ and $U(\lambda)= \bigoplus_{i=1}^{t(g_1)} \, \tfrac{F[X]}{X-\lambda} u_i$, the map $C(g,\lambda)$ is described by a matrix \[ A(g, \lambda) \in \M_{t(g_2) \times t(g_1)}(\tfrac{F_m[X]}{X-\lambda}) \simeq \M_{t(g_2) \times t(g_1)}(F_m), \qquad C(g, \lambda) v_j = \sum_i A(g,\lambda)_{i,j} u_i.\] The submatrix $A(g,\lambda)_{\alpha, \beta} \in \M_{r' \times r'}(F_m)$ of 
$A(g, \lambda)$ on the first $r'$ rows and columns 
describes the map $C(g,\lambda)_{\alpha, \beta}$. The matrix $g_3 \in \M_{n-k \times k}(F)$, defines a linear map $F^{n-k} \xrightarrow{g_3} F^k$. Consider the map 
\[g_3' \in \Hom_F(K(g_2), \text{coker}(p(X))), \; \text{defined by }\;  K(g_2) \hookrightarrow F^{n-k} \xrightarrow{g_3} F^k \to \tfrac{F^k}{p(X) F^k}.\]
We also consider the linear map 
\[ g_3(\alpha, \beta) \in \Hom_{F}(\beta, F^k/\alpha), \; \text{defined  by } \; \beta \hookrightarrow K(g_2) \xrightarrow{g_3'} \tfrac{F^k}{p(X) F^k} \to \tfrac{F^k}{\alpha}.\]
Extending the scalars  from $F$ to $F_m$, we consider $g_3'$ as an element of  
\[ g_3' \in \Hom_{F_m}(K(g_2)\otimes F_m, \tfrac{F_m^k}{p(X) F_m^k}) = \bigoplus_{j=0}^{m-1} \bigoplus_{i=0}^{m-1}  \Hom_{F_m}(K(g_2, \phi^j(\lambda)), U(\phi^i(\lambda))).\]
There are maps $g_3'(\phi^i(\lambda), \phi^j(\lambda)) \in \Hom_{F_m}(K(g_2, \phi^j(\lambda)), U(\phi^i(\lambda)))$ such that
\[ g_3'(\phi^i(\lambda), \phi^j(\lambda)) =  \phi^j \circ g_3'(\phi^{i-j}(\lambda), \lambda) \circ \phi^{-j},\] and hence is completely determined by $g_3'(\phi^\ell(\lambda), \lambda)$ for $\ell \in \{0, m-1\}$.  Again in terms of  the generators $K(g_2,\lambda) = \bigoplus_{j=1}^{t(g_2)} \tfrac{F[X]}{X-\lambda} v_j$ and $U(\lambda)= \bigoplus_{i=1}^{t(g_1)} \, \tfrac{F[X]}{X-\lambda} u_i$, the maps $g_3'(\phi^\ell(\lambda), \lambda)$ for $\ell \in \{0, m-1\}$ are described by matrices $B(g, \phi^\ell(\lambda))$ defined by $g_3'(\phi^\ell(\lambda), \lambda)(v_j)=\sum_i B(g, \phi^\ell(\lambda)) u_i$. 
Similarly, the map $g_3'(\alpha, \beta) \in \Hom_{F_m}( \beta \otimes F_m, F_m^k/(\alpha \otimes F_m))$ is completely described by the associated maps  $g_3'(\alpha, \beta)(\phi^\ell(\lambda), \lambda)$ for $\ell \in \{0, m-1\}$. The 
submatrix $B(g,\phi^\ell(\lambda))_{\alpha, \beta} \in \M_{r' \times r'}(F_m)$ of 
$B(g, \phi^\ell(\lambda))$ on the first $r'$ rows and columns of $B(g,\phi^\ell(\lambda))$
describes the map $g_3'(\alpha,\beta)(\phi^\ell(\lambda), \lambda)$.
The next lemma relates the matrices $A(g,\lambda)$ and $ B(g, \lambda) \in \M_{t(g_2) \times t(g_1)}(F_m)$.

\begin{lemma} \label{lem_C(g)_g_3}
   $A(g, \lambda)= c_\lambda   B(g, \lambda) \, \in \M_{t(g_2) \times t(g_1)}(F_m)$ where $c_\lambda \in F_m^\times$ equals $\lambda \prod_{i=1}^{m-1}(\lambda- \phi^i(\lambda))$.     Considering the submatrices on the first $r'$ rows and columns of both matrices, we get \[A(g, \lambda)_{\alpha, \beta}= c_\lambda   B(g, \lambda)_{\alpha, \beta} \in \M_{r' \times r'}(F_m).\]
\end{lemma}

\begin{proof} 
 The proof is obtained by chasing the diagram \eqref{eq:snake}.   We lift $v \in K(g_2, \lambda) \subset F_m^{n-k}$ to $(0,v) \in F_m^n$, and then apply $p(g)$ to  $(0,v)$.  Writing $p(X)=\prod_{i=1}^{m} (X-\phi^{m-i}(\lambda))$, we see  $p(g)=\bbsm p(g_1) & p(g)_3 \\ & p(g_2) \besm$ where 
\[ p(g)_3 = \sum_{i=0}^{m-1} 
(g_1-\phi^{m-1}(\lambda)) \dots (g_1-\phi^{m-i}(\lambda)) \, g_1g_3 \, (g_2-\phi^{m-i-2}(\lambda)) \dots (g_2-\lambda).\]
Therefore, for $v \in K(g_2,\lambda)$, we have 
\[p(g) \bbsm 0 \\ v \besm = \bbm  (g_1-\phi^{m-1}(\lambda)) \dots (g_1-\phi(\lambda))g_1g_3 v \\ 0 \bem. \]
We then have $C(g,\lambda) v \in U(\lambda)$ equals the projection of $(g_1-\phi^{m-1}(\lambda)) \dots (g_1-\phi(\lambda))g_1g_3 v \in F_m^k$ on $U(\lambda)$. Since $g_1$ acts as multiplication by $\lambda$ on $U(\lambda)$, we have 
\[ C(g, \lambda) v = c_\lambda \, g_3 v, \quad  c_\lambda = \lambda \prod_{i=1}^{m-1}(\lambda- \phi^i(\lambda)).\]
Since the matrix $A(g, \lambda)$ and $B(g,\lambda)$ represent the map 
$C(g,\lambda) \in \Hom_{F_m[X]}(K(g_2, \lambda), U(\lambda))$, and  
$g_3' \in \Hom_{F_m}(K(g_2,\lambda), U(\lambda))$ respectively, the proof is complete.
 \end{proof}
 For  $h \in \Hom_{F[X]}(\beta, F^k/\alpha)$ invertible, the inverse map $h^{-1}$ viewed as an element of $\Hom_{F_m[X]}( F_m^k/(\alpha \otimes F_m), \beta \otimes F_m)$ is of the form $\oplus_{i=0}^{m-1} \alpha'(\phi^i(\lambda)) \xrightarrow{h^{-1}_{\phi^i(\lambda)}} \oplus_{i=0}^{m-1} \beta(\phi^i(\lambda))$ with 
 $h^{-1}_{\phi^i(\lambda)} = \phi^i \circ h^{-1}_\lambda \circ \phi^{-i}$. We now consider the trace of the composite map $
 \beta \xrightarrow{g_3(\alpha, \beta)}  F^k/\alpha \xrightarrow{h^{-1}} \beta$. We have  
 \[ \tr(h^{-1} g_3(\alpha, \beta))= \sum_{i=0}^{m-1} \tr( h^{-1}_{\phi^i(\lambda)} g_3'(\alpha,\beta)(\phi^i(\lambda),\phi^i(\lambda))) = \tr_{F_m/F} \tr(h^{-1}_\lambda g_3'(\alpha,\beta)(\lambda,\lambda)). \]
 The map $g_3'(\alpha,\beta)(\lambda, \lambda)$ is represented by the matrix 
$B(g,\lambda)_{\alpha, \beta}$ which equals
$c_\lambda^{-1} A(g,\lambda)_{\alpha, \beta}$ by Lemma  \ref{lem_C(g)_g_3}. \\

We recall that $A(g,\lambda)_{\alpha, \beta} \in \M_{r' \times r'}(F_m)$ represents  the map $C(g,\lambda)_{\alpha, \beta}$.
Let $h_0 \in \GL_{r'}(F_m)$ denote the matrix of $h^{-1}_\lambda$ with respect to the generators $u_1, \dots, u_{r'}$ of 
$\alpha'(\lambda)= \oplus_{i=1}^{r'} \, \tfrac{F[X]}{X-\lambda} u_i$ 
and $v_1, \dots, v_{r'}$ of 
$\beta(\lambda) = \oplus_{j=1}^{r'} \tfrac{F[X]}{X-\lambda} v_j$. We now have 

\[ \tr(h^{-1} g_3(\alpha,\beta))=\tr_{F_m/F} \tr (c_\lambda^{-1}  h_0   A(g,\lambda)_{\alpha, \beta}).\]
We note that $\tilde \psi_0=\psi_0 \circ \tr_{F_m/F}$ is a non-trivial additive character of $F_m$, and hence  we have by Theorem \ref{LAA_thm}:
\begin{multline*}
     \sum_{h \in \text{Hom}_{F[X]}(\beta, F^k/\alpha)} \psi_0( \tr(h^{-1} g_3(\alpha,\beta) ))=
\sum_{h_0 \in \GL_{r'}(F_m)} \tilde \psi_0 ( \tr c_\lambda^{-1} h_0^{-1} A(g,\lambda)_{\alpha, \beta})\\=  (-1)^s (q')^{\binom{s}{2} +s(r'-s)} |\GL_{r'-s}(F_m)|, \end{multline*}
where $s$ is the rank of $C(g)_{\alpha, \beta}$.


\noindent{\textbf{\underline{Step 5:} Simplification of \eqref{eq:sum_over_h}} \hfill \\

Using the expression \eqref{eq:n_r_s} for $n_{r}(s)$    in \eqref{eq:sum_over_h} and 
summing this over $r \geq 1$, we see that the character of $\sum_{r \geq 1} \Ind_{P_{k,n}(r)}^{P_{k,n}} ( \psi_r \otimes [\St_{G_r} \otimes \Pi_{r,n-k}])$ at $g$ is:
\begin{multline} \label{eq:sum_r_t_s}
    (-1)^{n-k} \left( \sum_{r \geq 1}  (q')^{\binom{r'}{2}} ((q')^r;q')_{t(g_2)-r'}  \right) \cdot \left(\sum_t (-1)^{r'-t} \stirling{t(g_3)}{r'-t}_{q'} \stirling{t(g_2)-t(g_3)}{t}_{q'} (q')^{t(t(g_3)-r'+t)} \right) \\
    \cdot \left( \sum_s (-1)^{t-s} (q')^{\binom{s}{2} +s(r'-s)} |\GL_{r'-s}(F_m)|  \stirling{t}{t-s}_{q'} \stirling{t(g_1)-t}{t(g_1)-r'-t+s}_{q'}  (q')^{(t(g_1)-r'-t+s)s} \right).
\end{multline}

 Let $I_s$  denote the third parenthetical term in the above expression.  We will show that  $I_s$ equals  
$\stirling{t(g_1)-t}{r'-t}_{q'} |\GL_{r'-t}(F_m)| (q')^{t(r'-t) + \binom{t}{2}}$. To do so, we first write
\[\stirling{t(g_1)-t}{t(g_1)-r'-t+s}_q  |\GL_{r'-s}(F_m)|=(q')^{\binom{r'-s}{2}} \prod_{i=0}^{r'-s-1} ((q')^{t(g_1)-t-i}-1),  \]
using which the term $I_s$  can be rewritten as
\[ (q')^{\binom{r'}{2}} (\prod_{i=t}^{r'-1} ((q')^{t(g_1)-i}-1) ) \sum_s (-1)^{t-s} \stirling{t}{t-s}_{q'} (q')^{s(t(g_1)-r'-t+s)}    \prod_{i=0}^{t-s-1} ((q')^{t(g_1)-r'-i}-1),\]
which in turn can be rewritten as:
\[ (q')^{\binom{r'}{2}+t(t(g_1)-r') -\binom{t}{2} } (\prod_{i=t}^{r'-1} ((q')^{t(g_1)-i}-1) ) \sum_s (-1)^{t-s} \stirling{t}{t-s}_{q'} (q')^{\binom{s}{2}} ((q')^{r'-t(g_1)};q')_{t-s}.\]
The expression $\sum_s (-1)^{t-s} \stirling{t}{t-s}_{q'} (q')^{\binom{s}{2}} ((q')^{r'-t(g_1)};q')_{t-s}$ above, equals $(q')^{\binom{t}{2}} (q')^{(r'-t(g_1))t}$ by Lemma \ref{id_T}. Therefore the term $I_s$ equals
\begin{multline*}
    (q')^{\binom{t}{2}+(r'-t(g_1))t} (q')^{\binom{r'}{2}+t(t(g_1)-r') -\binom{t}{2} } (\prod_{i=t}^{r'-1} ((q')^{t(g_1)-i}-1) ) \\= (q')^{\binom{r'}{2} - \binom{r'-t}{2}} |\GL_{r'-t}(F_m)| \prod_{i=1}^{r'-t} (\tfrac{(q')^{t(g_1)-t-(i-1)}-1}{(q')^i-1} ),
\end{multline*} 
which can be rewritten as 
\[\stirling{t(g_1)-t}{r'-t}_{q'} |\GL_{r'-t}(F_m)| (q')^{t(r'-t) + \binom{t}{2}}, \]
which proves the claim about the term $I_s$. The expression \eqref{eq:sum_r_t_s}  can now be written as 
\begin{multline*}
    (-1)^{n-k} \left( \sum_{r' \geq 1}  (q')^{\binom{r'}{2}} ((q')^{r'};q')_{t(g_2)-r'}   \right) \cdot \\ \left(\sum_t (-1)^{r'-t} \stirling{t(g_3)}{r'-t}_{q'}    \stirling{t(g_1)-t}{r'-t}_{q'} |\GL_{r'-t}(F_m)| \stirling{t(g_2)-t(g_3)}{t}_{q'} (q')^{\binom{t}{2}+  t(g_3) t}\right). \end{multline*}
Writing 
\[ (q')^{\binom{r'}{2}} ((q')^{r'};q')_{t(g_2)-r'}= (-1)^{t(g_3)-r'+t} (q')^{\binom{t(g_3)+t}{2}}   \, ((q')^{t(g_3)+t};q')_{t(g_2)-t(g_3)-t} \, ((q')^{1-t(g_3)-t};q')_{t(g_3)-r'+t},\]
and using the fact that  $\stirling{t(g_3)}{r'-t}_{q'}   \stirling{t(g_1)-t}{r'-t}_{q'}  |\GL_{r'-t}(F_m)|=a(t(g_1)-t \times t(g_3),r'-t,q')$  we can rewrite \eqref{eq:sum_r_t_s} as
\begin{multline*}
    (-1)^{n-k} \left(  \sum_t (-1)^{t(g_3)} (q')^{\binom{t(g_3)+t}{2} + \binom{t}{2}+  t(g_3) t}((q')^{t(g_3)+t};q')_{t(g_2)-t(g_3)-t} \stirling{t(g_2)-t(g_3)}{t}_{q'} \right)\\  \cdot\left(\sum_{r'} a(t(g_1)-t \times t(g_3),r'-t,q')  ({q'}^{1-t(g_3)-t};q')_{t(g_3)-r'+t} \right). \end{multline*}
By Lemma \ref{amnr_lemma}, the second term in the product above is 
\[ (q')^{t(g_3)(t(g_1)-t)}\, ((q')^{1-t(g_3)-t(g_1)};q')_{t(g_3)}= (-1)^{t(g_3)} (q')^{-t(g_3)  t - \binom{t(g_3)}{2}} ((q')^{t(g_1)}; q')_{t(g_3)} .\] Therefore, we can rewrite \eqref{eq:sum_r_t_s} as
 
\begin{multline*}
    (-1)^{n-k}  (q')^{\binom{t(g_2)}{2} - \binom{t(g_3)} {2}} ((q')^{t(g_1)}; q')_{t(g_3)} \cdot \\  \left(  \sum_t (-1)^{t(g_2)-t(g_3)-t}  \stirling{t(g_2)-t(g_3)}{t(g_2)-t(g_3)-t}_{q'}  ((q')^{1-t(g_2)};q')_{t(g_2)-t(g_3)-t} (q')^{\binom{t}{2}}\right).  \end{multline*}
The last parenthetical term above equals 
\[(q')^{(1-t(g_2))(t(g_2)-t(g_3))} (q')^{\binom{t(g_2)-t(g_3)}{2}} = (q')^{\binom{t(g_3)}{2} - \binom{t(g_2)} {2}}, \]
by Lemma \ref{id_T}. Therefore  \eqref{eq:sum_r_t_s} equals
\[  (-1)^{n-k} ((q')^{t(g_1)};q')_{t(g_3)}= (-1)^{n-k} ((q')^{t(g_1)};q')_{t(g)-t(g_1)} \]
as was to be shown. This completes the proof of Theorem \ref{Pi_k_n_char}. 
\qed

\subsection{Universal property of $\Pi_{k,n}$ with respect to cuspidal representations} \hfill \\   The representation $(\Phi^+)^{n-1}(1)$ of  $\Mir_n$  satisfies a well known universal property with respect to restriction to $\Mir_n$ of cuspidal representations of $G_n$ (for example \cite[5.18]{BerZel[2]}): 
\[ \pi_\theta|_{\Mir_n}= (\Phi^+)^{n-1}(1) .\] We can express this in terms of the representation $\Pi_{1,n}$. Using  the representation $\Pi_{1,n}^\dagger = 1_{F^\times} \boxtimes (\Phi^+)^{n-1}(1)$ of $P_{n-1,n}=F^\times  \times \Mir_n$ (from Lemma \ref{Pi_1_n_Mir}), we can rewrite this universal property as 
\[ \pi_\theta|_{P_{n-1,n}}=  \Pi_{1,n}^\dagger \otimes \theta|_{F^\times},\]
where  $\theta|_{F^\times}(g)=\theta(g_{n,n})$.
Applying $\dagger$ to this, we get 
\[ \pi_\theta^\dagger| _{P_{1,n}}= \Pi_{1,n} \otimes \bar\theta|_{F^\times}  ,\]
where $\bar\theta|_{F^\times}(g)=\bar\theta(g_{1,1})$. Using  \eqref{eq:pi_theta_dagger} in this equation,  we get $\pi_{\bar \theta} |_{P_{1,n}}= \Pi_{1,n} \otimes \bar \theta|_{F^\times}$. Finally, replacing $\bar\theta$ with $\theta$, we get 
\[     \pi_{\theta} |_{P_{1,n}}= \Pi_{1,n} \otimes  \theta|_{F^\times}.\]
In summary, we have 
\begin{equation} \label{eq:univ_Pi_1_n}
   \pi_\theta|_{P_{n-1,n}}=  \Pi_{1,n}^\dagger \otimes \theta|_{F^\times}, \qquad \pi_{\theta} |_{P_{1,n}}= \Pi_{1,n} \otimes  \theta|_{F^\times}.
 \end{equation}

Parts (2)-(3) of the  next theorem generalize these properties \eqref{eq:univ_Pi_1_n} to $P_{k,n}$, and are consequences of part (1) of the theorem, which is much more general and applies not only to cuspidal representations but also to the class functions $\Theta_\theta$:

  \begin{theorem} \label{Pi_universal_property}
        Let $P_{k,n}$ be a maximal parabolic subgroup of $G_n$. For $g  = \bbsm g_1 & g_1 g_3 \\0 & g_2 \besm \in P_{k,n}$, let $p_1 \colon P_{k,n} \to G_k$ be the epimorphism $p_1(g)=g_1$, and let $p_2 \colon P_{k,n} \to G_{n-k}$ be the epimorphism $p_2(g)=g_2$. Let $d=\text{gcd}\{k,n\}$.
     \begin{enumerate}
         \item[\emph{(1)}] Let $\theta$ be any character of $F_n^\times$, and let 
         $\theta_k$ be any character of $F_k^\times$ whose restriction to $F_d^\times$ is $\theta|_{F_d^\times}$ (such a $\theta_k$ exists by  part $(1)$ of Lemma \ref{imath_lemma}). Let $\Theta_\theta$ and $\Theta_{\theta_k}$ be the associated class functions on $G_n$ and $G_k$ respectively (see Definition \ref{Theta_definition}). We recall that  $\Theta_{k,n}$ denotes  the character of $\Pi_{k,n}$. For $g \in P_{k,n}$, we have:
         \[  \Theta_\theta(g)= \Theta_{k,n}(g) \cdot \Theta_{\theta_k}(g_1).\]
    
\item[\emph{(2)}] Let $\theta$ be a regular character of $F_n^\times$ and let $\pi_\theta$ be the associated cuspidal representation of $G_n$. If  either $k \nmid n$, or if $k\mid n$  and $\theta|_{F_k^\times}$ is a regular character of $F_k^\times$, there is a cuspidal representation $\pi_{\tilde\theta}$ of $G_k$ such that 
    \[ \pi_\theta|_{P_{k,n}} \simeq \Pi_{k,n} \otimes \pi_{\tilde \theta},
    \]
    where the representation $\pi_{\tilde\theta}$ of $G_k$ is inflated to $P_{k,n}$ via $p_1$. In the former case $k \nmid n$, we can take $\tilde\theta$ (by part $(2)$ of Lemma \ref{imath_lemma})  to be a regular character of $F_k^\times$ such that  $\tilde\theta|_{F_d^\times} = \theta|_{F_d^\times}$. In the latter case $\tilde \theta = \theta|_{F_k^\times}$. \\
 
\item[\emph{(3)}] Let $\pi_\theta$ be as in part (2) above. If  either $(n-k) \nmid n$, or if $(n-k) \mid n$ and $\theta|_{F_{n-k}^\times}$ is a regular character of $F_{n-k}^\times$, there is a cuspidal representation $\pi_{\tilde\theta}$ of $G_{n-k}$ such that 
    \[ \pi_\theta|_{P_{k,n}} \simeq \Pi_{n-k,n}^\dagger \otimes \pi_{\tilde \theta},
    \]
    where the representation $\pi_{\tilde\theta}$ of $G_{n-k}$ is inflated to $P_{k,n}$ via $p_2$.  In the former case $(n-k) \nmid n$, we can take $\tilde\theta$ (by part $(2)$ of Lemma \ref{imath_lemma})  to be a regular character of $F_{n-k}^\times$ such that  $\tilde\theta|_{F_d^\times} = \theta|_{F_d^\times}$. In the latter case $\tilde \theta = \theta|_{F_{n-k}^\times}$.

    \end{enumerate}
\end{theorem}
\begin{proof} (1) 
Let $g = s u$ be decomposition of $g=\bbsm g_1 & g_1g_3\\ & g_2 \besm$  into its semisimple and unipotent parts, with $s=\bbsm s_1 & s_3\\ &s_2 \besm$ and $u= \bbsm u_1  &u_3\\ & u_2 \besm$. By Lemma \ref{jordan_chevalley_for_P_k_n},  $s_1$ and $s_2$ are the semisimple parts of $g_1$ and $g_2$ respectively.  If $s_1$  does not come from $F_k^\times$, then $s$ also does not come from $F_n^\times$, and hence  $\Theta_\theta(g)$ and $\Theta_{\theta_k}(g_1)$ are both zero by Definition \ref{Theta_definition}, which verifies  $\Theta_\theta(g)= \Theta_{k,n}(g) \cdot \Theta_{\theta_k}(g_1)$. We now assume $s_1$ comes from $F_k^\times$. In this case if $s$ does not come from $F_n^\times$, then $\Theta_{k,n}(g)=0$ by Theorem \ref{Pi_k_n_char}, and $\Theta_\theta(g)=0$ by Definition \ref{Theta_definition}, which again verifies  $\Theta_\theta(g)= \Theta_{k,n}(g) \cdot \Theta_{\theta_k}(g_1)$. We now assume $s_1$ comes from $F_k^\times$ and $s$ comes from $F_n^\times$ (and hence both come from $F_d^\times$). If $\lambda \in F_m$ is a root of the minimal polynomial $p(X)$ of $s$, then by Theorem \ref{Pi_k_n_char} and Definition \ref{Theta_definition} we have
\begin{multline*} \Theta_{k,n}(g) \cdot \Theta_{\theta_k}(g_1)= \left[ (-1)^{n-k} ({q'}^{t(g_1)};q')_{t(g)-t(g_1)}\right]  \cdot 
\left[(-1)^{k-1} \left( \sum_{i=0}^{m-1} \theta(\phi^i(\lambda)) \right) \,(q';q')_{t(g_1)} \right] \\ = (-1)^{n-1} \left( \sum_{i=0}^{m-1} \theta(\phi^i(\lambda)) \right) \, (q';q')_{t(g)} = \Theta_\theta(g).
\end{multline*}

(2) This simply follows from part (1), by noting that $\Theta_{\theta_k}$ is the character of the cuspidal representation $\pi_{\tilde\theta}$ of $G_k$. \\

(3) By part (2), there is a cuspidal representation $\pi_{\tilde\theta}$ of $G_{n-k}$ such that  $\pi_\theta|_{P_{n-k,n}} \simeq \Pi_{n-k,n} \otimes \pi_{\tilde \theta}$. From this we get
\[ \pi_\theta^\dagger|_{P_{k,n}} \simeq \Pi_{n-k,n}^\dagger \otimes \pi_{\tilde \theta}^\dagger.\]
Using \eqref{eq:pi_theta_dagger}, we get 
\[ \pi_{\bar\theta}|_{P_{k,n}} \simeq \Pi_{n-k,n}^\dagger \otimes \pi_{\widetilde{\bar\theta}},\]
Finally, replacing the regular character $\bar\theta$  with $\theta$, we get 
\[ \pi_{\theta}|_{P_{k,n}} \simeq \Pi_{n-k,n}^\dagger \otimes \pi_{\tilde\theta},\]
where the cuspidal representation $\pi_{\tilde\theta}$ of $G_{n-k}$ is inflated to $P_{k,n}$ via $p_2$.\\


\end{proof}
The inverse transpose automorphism $g \mapsto g^{-\top}$ carries the subgroup $P^*_{k,n}$ of $G_n$ consisting of matrices of the form $g=\bbsm g_1 & \\g_5g_1 & g_4 \besm$ to the parabolic subgroup $P_{k,n}$. This allows us to define:
\begin{definition} \label{Pi^*_def}
    We define a representation $\Pi^*_{k,n}$ of $P^*_{k,n}$ by 
    \[ \Pi^*_{k,n}(g)=\Pi_{k,n}(g^{-\top}).\]
\end{definition}
The character of $\Pi^*_{k,n}$ will be denoted $\Theta^*_{k,n}$. The results of Theorem
\ref{Pi_universal_property} have obvious analogues for the representation $\Pi^*_{k,n}$ of $P^*_{k,n}$. We only record here the first property. Let $\theta$ and $\theta_k$  be as in part (1) of Theorem \ref{Pi_universal_property}. For $g=\bbsm g_1 & \\g_5 g_1 & g_4 \besm \in P^*_{k,n}$, we have: 
         \begin{equation} \label{eq:Pi*_universal_property}
         \Theta_\theta(g)= \Theta^*_{k,n}(g) \cdot \Theta_{\theta_k}(g_1).
\end{equation}

\section{Structure  of $\pi_{N,\psi}$} \label{structure}
As before, let $\pi_\theta$ be a 
cuspidal representation of $G_n$, and $P=MN$  a maximal parabolic subgroup of $G_{n}$ associated with 
the partition $(k,n-k)$ of $n$, and 
$\psi$ a rank $r$ character  of the unipotent radical $N \simeq \M_{k \times n-k}(F)$. In this section we determine the structure of the twisted Jacquet module $\pi_{N, \psi}$ for general $r$.   We recall from \eqref{eq:Mpsi} that $M_{\psi}$ (which we denote $M_r$ here)  equals:
\begin{multline*} \label{eq:Mpsi}
    M_r=\bigg\{ \bbsm g_1&0\\0&g_2 \besm \colon g_1=\bbsm h_1 &  0\\
    h_5 & h_4\besm, \; g_2=\bbsm
    h_1& h_3\\0&h_2 \besm, \; h_1\in G_r, \, h_4\in G_{k-r},\\ h_2\in G_{n-k-r}, \,   
    h_5\in \M_{k-r \times r}(F), \, h_3\in \M_{r \times n-k-r}(F)
    \bigg\}.\end{multline*}
For $g=\bbsm g_1&0\\0&g_2 \besm \in M_r$ with 
$g_1=\bbsm h_1 &  0\\
    h_5 & h_4\besm$ and $g_2=\bbsm
    h_1& h_3\\0&h_2 \besm$, let 
    \[ M_r \xrightarrow{p_1} P^*_{r,k}, \quad M_r \xrightarrow{p_2} P_{r,n-k}, \quad  M_r \xrightarrow{p_3} G_r,\]
    be the epimorphisms defined by 
    \[ p_1(g)=g_1, \quad  p_2(g)=g_2, \quad  p_3(g)=h_1.\]
    Let $d=\text{gcd}(r,k,n)$ and let $\theta_r$ be a character of $F_r^\times$ with the property that $\theta_r|_{F_d^\times} = \theta|_{F_d^\times}$ (such a $\theta_r$ exists by  part (1) of Lemma \ref{imath_lemma}). Consider the representation $\Ind_{F_r^\times}^{G_r} \theta_r$ of $G_r$. The inflation of this representation to $M_r$ via the epimorphism $p_3$ will be denoted by the same symbol $\Ind_{F_r^\times}^{G_r} \theta_r$. Let $\Pi_{r,n-k}$ and $\Pi^*_{r,k}$ be the representations of $P_{r,n-k}$ and $P^*_{r,k}$ respectively, defined in Section \ref{section_Pi}. The inflations of $\Pi_{r,n-k}$ and $\Pi^*_{r,k}$ to $M_r$ via the epimorphisms $p_2$ and $p_1$ will be denoted by the same symbols. 
\begin{theorem} \label{main_theorem}
Let $\pi=\pi_\theta$ be a cuspidal representation of $G_n$, and let $P$ be a maximal parabolic subgroup of $G$ associated with the partition $(k,n-k)$ of $n$. Let $P=MN$ be the Levi decomposition of $P$, and let $\psi$ be a rank $r$ character of $N$. Let $\theta_r$ be a character of $F_r^\times$ as above. The  twisted Jacquet module $\pi_{N, \psi}$ as a representation of $M_r$ is isomorphic to 
\[ \pi_{N, \psi} \simeq \Pi^*_{r,k} \otimes \Pi_{r,n-k} \otimes \Ind_{F_r^\times}^{G_r} \theta_r. \] 
In particular,
\[ \dim  \pi_{N,\psi} =
q^{\binom{r}{2}} \prod_{i=1}^{r-1}(q^i-1) \cdot \prod_{i=r}^{k-1}(q^i-1) \cdot \prod_{i=r}^{n-k-1}(q^i-1).\]
\end{theorem}
\begin{proof}
Let $\Theta_{N, \psi}$ denote the character of  twisted Jacquet module $\pi_{N, \psi}$ of $\pi_\theta$ for a rank $r$ character $\psi$ of $N$.
We recall that the character of $\pi_\theta$ equals the class function $\Theta_\theta$ of Definition \ref{Theta_definition}.
By part (1) of Theorem \ref{Pi_universal_property}, for $g' = \bbsm g_1 & g_1g_3 \\& g_2 \besm \in P_{k,n}$, we have \[ \Theta_{\theta}(g')=\Theta_{k,n}(g') \cdot \Theta_{\theta_k}(g_1), \]
where we recall that $\theta_k$ is any character of $F_k^\times$ such that the restrictions to $F_{\text{gcd}\{k,n\}}^\times$ of $\theta_k$ and $\theta$ are equal, and where $\Theta_{\theta_k}$ is the  class function (see Definition \ref{Theta_definition}) associated to $\theta_k$.  Using this in Proposition \ref{character_Jacquet} (with a rank $r$ character $\psi_r$ of $N$) for calculating  the character $\Theta_{N, \psi}$
at $g=\bbsm g_1 & \\ & g_2 \besm \in M_r$ with $g_1=\bbsm h_1 &  0\\
    h_5 & h_4\besm$ and $g_2=\bbsm
    h_1& h_3\\0&h_2 \besm$, we get:
\[ \Theta_{N, \psi}(g) = \Theta_{\theta_k}(g_1) \cdot \chi_{ (\Pi_{k,n})_r}(g). \]
 By Definition \ref{def_Pi}, we have $ (\Pi_{k,n})_r= \St_{G_r} \otimes \Pi_{r,n-k}$, and hence we get
\[ \Theta_{N, \psi}(g) = \left[ \Theta_{\theta_k}(g_1) \cdot \St_{G_r}(h_1) \right] \cdot \Theta_{r,n-k}(g_2). \]
Applying the result in equation \eqref{eq:Pi*_universal_property} with $(k,n, \text{gcd}\{k,n\},\theta) \mapsto (r,k,\text{gcd}\{r,k\}),\theta_k)$, we see that for $\tilde\theta$ any  character  of $F_r^\times$ such that the restrictions to $F_{\text{gcd}\{r,k\}}^\times$ of $\tilde\theta$ and $\theta_k$ are equal, we have
\[    \Theta_{\theta_k}(g_1)= \Theta^*_{r,k}(g_1) \cdot \Theta_{\tilde\theta}(h_1).\]
Using this in the previous equation, we get:
\[ \Theta_{N, \psi}(g) = \left[ \Theta_{\tilde\theta}(h_1) \cdot \St_{G_r}(h_1) \right] \cdot \Theta^*_{r,k}(g_1)  \cdot \Theta_{r,n-k}(g_2). \]
By part (2) of Lemma \ref{ind=st_times_cuspidal}, we can rewrite this as
\[ \Theta_{N, \psi}(g) = \chi_{\Ind_{F_r^\times}^{G_r} \tilde\theta}(h_1)   \cdot \Theta^*_{r,k}(g_1)  \cdot \Theta_{r,n-k}(g_2). \]
The character of $\Ind_{F_r^\times}^{G_r} \beta$ at $h_1 \in G_r$ for any character $\beta$ of $F_r^\times$, is zero unless $h_1$ comes from $F_r^\times$. Let $m|r$ be the degree of the   minimal polynomial $p(X)\in F[X]$ of $h_1$. Since $g_2=\bbsm
    h_1& h_3\\0&h_2 \besm$, we have by Theorem \ref{Pi_k_n_char} (which applies as the semisimple part of $h_1$, which is $h_1$ itself, does come from $F_r^\times$), that $\Theta_{r,n-k}(g_2)=0$ unless the semisimple part of $g_2$ comes from $F_{n-k}^\times$. By part (2) of Lemma \ref{jordan_chevalley_for_P_k_n}, this implies that the degree $m$ of $p(X)$ divides $\text{gcd}\{r,n-k\}$. Repeating this argument with 
    $g_1=\bbsm h_1 &  0\\
    h_5 & h_4\besm$ and $\Theta^*_{r,k}(g_1)$, we see that $m| \text{gcd}\{r,k\}$, and hence $m$ divides $d=\text{gcd}\{r,k,n\}$.
In particular, the value of  the character of $\Ind_{F_r^\times}^{G_r} \tilde\theta$ at $h_1$ only depends on $ \tilde\theta|_{F_d^\times}$. Since $F_d \subset F_{\text{gcd}\{r,k\}}$,  and by definition of $\tilde\theta$, the restrictions to $F_{\text{gcd}\{r,k\}}^\times$ of $\tilde\theta$ and $\theta$ are equal, we see that $ \tilde\theta|_{F_d^\times}=\theta|_{F_d^\times}=\theta_r|_{F_d^\times}$, and hence, we may replace 
$\Ind_{F_r^\times}^{G_r} \tilde\theta$ with $\Ind_{F_r^\times}^{G_r} \theta_r$. Using this in the equation above, we conclude 
\[ \pi_{N, \psi} \simeq \Pi^*_{r,k} \otimes \Pi_{r,n-k} \otimes \Ind_{F_r^\times}^{G_r} \theta_r.\]
\end{proof}





\providecommand{\bysame}{\leavevmode\hbox to3em{\hrulefill}\thinspace}
\providecommand{\MR}{\relax\ifhmode\unskip\space\fi MR }
\providecommand{\MRhref}[2]{%
  \href{http://www.ams.org/mathscinet-getitem?mr=#1}{#2}
}
\providecommand{\href}[2]{#2}

\end{document}